\documentclass[11pt]{amsart}

\usepackage{color}

\usepackage{amsmath,amssymb,amsthm,amsfonts,amsbsy,amstext}

\usepackage{amsmath, amssymb} 
\usepackage{amsthm, amsfonts, mathrsfs}

\usepackage{amsfonts,graphicx}

\usepackage{fullpage}

\newcommand{\DD}{\textnormal{D}}

\theoremstyle{plain}
\newtheorem{Theo}{Theorem}[section]
\newtheorem{lem}[Theo]{Lemma}

\newtheorem{prop}[Theo]{Proposition}

\theoremstyle{plain} \theoremstyle{definition}
\newtheorem{defi}[Theo]{Definition}

\theoremstyle{remark}
\newtheorem{Rema}[Theo]{Remark}
\newtheorem*{rema*}{Remark}
\newtheorem*{remas}{Remarks}

\newcommand{\NN}{\mathbb{N}}

\newcommand{\RR}{\mathbb{R}}

\numberwithin{equation}{section}

\date{}

\begin{document}

\title[ Global well-posedness for stratified N.-S. system]
{Inviscid limit for axisymmetric stratified Navier-Stokes system}
\author[Samira Sulaiman]{Samira Sulaiman}
\address{IRMAR, Universit\'e de Rennes 1\\ Campus de
Beaulieu\\ 35~042 Rennes cedex\\ France}
\email{samira.sulaiman@univ-rennes1.fr}

\begin{abstract}
This paper is devoted to the study of the Cauchy problem for the stratified Navier-Stokes  system in space dimension three. In the first part of the paper, we prove the existence of a unique global solution $(v_\nu,\rho_\nu)$ for this system with axisymmetric initial data belonging to the Sobolev spaces $H^{s}\times H^{s-2}$ with $s>\frac{5}{2}.$ The bounds of the solution are uniform  with respect to the viscosity.
In the second part, we analyse the inviscid limit problem. We prove the strong convergence in the space $L^{\infty}_{\text{loc}}(\RR_+; H^{s}\times H^{s-2})$ of the viscous solutions $(v_\nu,\rho_\nu)_{\nu>0}$  to the solution $(v,\rho)$ of the stratified Euler system. 

\end{abstract}
\maketitle
\tableofcontents

\section{Introduction and main results}\label{A1}
In this paper, we consider the incompressible stratified Navier-Stokes  system in space dimension three
\begin{equation}\label{a} 
\left\{ \begin{array}{ll} 
\partial_{t}v_{\nu}+v_{\nu}\cdot\nabla v_{\nu}-\nu\Delta v_{\nu}+\nabla p_{\nu}=\rho_{\nu}\,e_{z},\qquad (t,x)\in\RR_{+}\times\RR^{3}\\ 
\partial_{t}\rho_{\nu}+v_{\nu}\cdot\nabla\rho_{\nu}-\Delta\rho_{\nu}=0\\
\textnormal{div}\,v_{\nu}=0\\
(v_{\nu},\rho_{\nu})_{| t=0}=(v^{0},\rho^{0}). 
\end{array} \right.
\end{equation} 
Here, the vector field $v_{\nu}=(v_{\nu}^1,v_{\nu}^2,v_{\nu}^3)$ stands for the velocity of the fluid and it is assumed to be  divergence-free, the scalar function $\rho_{\nu}$ denotes the density and  $p_{\nu}$ is the pressure. The \mbox{parameter  $\nu>0$} is the kinematic viscosity and the vector $e_{z}$ is given by $(0,0,1).$ Remark that the usual incompressible Navier-Stokes system arises as a particular case of \eqref{a}: it suffices to take $\rho=\textnormal{constant}$. It reads as follows:
\begin{equation}\label{b} 
\left\{ \begin{array}{ll} 
\partial_{t}v_{\nu}+v_{\nu}\cdot\nabla v_{\nu}-\nu\Delta v_{\nu}+\nabla p_{\nu}=0\\ 
\textnormal{div}\,v_{\nu}=0\\
{v_{\nu}}_{|t=0}=v^{0}.  
\end{array} \right.
\end{equation} 
The mathematical theory of the Navier-Stokes equations \eqref{b} was initiated by  Leray in \cite{JL2}. He proved the global existence of a weak global solution for the system  \eqref{b} in the energy space by using a compactness method. Nevertheless, the uniqueness of these solutions is only known in space dimension two. 
Few decades later, Fujita and Kato \cite{fk} proved the local well-posedness in the critical Sobolev space $\dot{H}^{\frac12}(\RR^3),$ by using a fixed point argument and taking benefit of the time decay of the heat semiflow. The global existence of these solutions is only proved  for small initial data and the question for large  data remains an outstanding open problem. For more discussion, we refer the reader for example to  the papers \cite{kt01,gl02,p96}. 

It seems that we can give a positive answer for the global existence when the initial data are not necessarily small but must have some special symmetry. Before going  further into the details, let us first write the equation of the vorticity  which plays a central  role in the theory of the  global well-posedness. For a given vector field $v,$ the vorticity $\omega$ is defined by  $\omega=\text{curl } v=\nabla\times v$ and in the case of the system \eqref{b} it  solves the transport-diffusion equation,
\begin{equation}\label{vort}
\partial_{t}\omega_{\nu}+v_{\nu}\cdot\nabla\omega_{\nu}-\nu\Delta\omega_{\nu}=\omega_{\nu}\cdot\nabla v_{\nu}.
\end{equation}
The main difficulty is related to  the dynamics of the stretching term $\omega_{\nu}\cdot\nabla v_{\nu}$ which is  very complex and generates a lot of unsolved problems . Now we will see how to use the  axisymmetry of the flows in order to get   a  better understanding  of  the stretching term. We start with the following definition:
\begin{defi}
We say that a vector field $v$ is axisymmetric (without swirl) if it takes the form:
\begin{equation*}
v(t,x)=v^{r}(t,r,z) e_{r}+v^{z}(t,r,z) e_{z},
\end{equation*}
where $z=x_3\;,\; x=(x_1, x_2, z)\;,\;r=(x_{1}^{2}+x_{2}^{2})^{\frac{1}{2}}\;\;\textnormal{and}\;\;(e_r, e_\theta, e_z)$ is the cylindrical basis of $\RR^3$ given by :
\begin{equation*}
e_r=\big(\frac{x_1}{r},\frac{x_2}{r},0\big)\quad e_\theta=\big(-\frac{x_2}{r}, \frac{x_1}{r},0\big)\quad\textnormal{and}\;\;\;e_z=(0,0,1).
\end{equation*}
The components $v^r$ and $v^z$ do not depend on the angular variable $\theta.$
\end{defi}
We  need in what follows to recall some basic algebraic properties related to some computations in the cylindrical coordinates system. For example, for an axisymmetric vector field $v$,   the operators $v\cdot\nabla$ and $\text{div}$ can be written under the form:
\begin{eqnarray}\label{f}
\nonumber v\cdot\nabla&=&v^{r}\partial_{r}+\frac{1}{r}v^{\theta}\partial_{\theta}+v^{z}\partial_{z}\\
&=& v^{r}\partial_{r}+v^{z}\partial_{z}
\end{eqnarray}
and
$$\textnormal{div}\,v= \partial_{r}v^{r}+\frac{v^r}{r}+\partial_{z}v^{z}.$$
The vorticity $\omega$ of the  vector field $v$ has the special form :
\begin{equation}\label{c}
\omega=(\partial_{z} v^{r}-\partial_{r} v^{z})e_{\theta} :=\omega^{\theta}e_{\theta}
\end{equation}
and the streching term reads $$\omega\cdot\nabla v=\frac{v^r}{r}\omega.$$
Consequently, the equation \eqref{vort} becomes
$$\partial_{t}\omega_{\nu}+v_{\nu}\cdot\nabla\omega_{\nu}-\nu\Delta\omega_{\nu}=\frac{v_{\nu}^{r}}{r}\omega_{\nu}.$$ 
The expression of the Laplacian operator in the cylindrical coordinates system is given by  {$\Delta=\partial_{rr}+\frac{1}{r}\partial_{r}+\partial_{zz}$.} Therefore, the scalar component  $\omega_{\nu}^{\theta}$ of the vorticity will satisfy the equation
$$\partial_{t}\omega_{\nu}^{\theta}+v_{\nu}\cdot\nabla\omega_{\nu}^{\theta}-\nu(\Delta\omega_{\nu}^{\theta} -\frac{\omega_{\nu}^{\theta}}{r^2})=\frac{v_{\nu}^{r}}{r}\omega_{\nu}^{\theta}.$$
We can easily chek that the quantity $\beta:=\frac{\omega_{\nu}^{\theta}}{r}$ solves the equation
$$\partial_{t}\beta+v\cdot\nabla\beta-\nu(\Delta+\frac{2}{r}\partial_{r})\beta=0.$$
This illustrates the fact that the advection term $v_{\nu}\cdot\nabla\omega_{\nu}^{\theta}$ has a smoothing effect and allows to kill the stretching term. Hence, we deduce that
for all $p\in[1,\infty],$
$$\Vert\beta(t)\Vert_{L^p}\le\Vert\beta^{0}\Vert_{L^p}.$$

These new conservation laws enable Ukhoviskii and  Iudovich \cite{ui68}, to get the global well-posedness under the assumption $v^{0}\in H^1$ and   $\omega_{0},\frac{\omega_{0}}{r}\in L^{2}\cap L^{\infty}.$ This result has been recently improved by many authors in various function spaces and where   the inviscid case $\nu=0$ is also treated, for more details see for example   \cite{a08,ahs08,d, hz09,lmn99, Taira}.

Concerning the inviscid limit problem, that is the convergence of the viscous solutions $(v_\nu)_{\nu>0}$ to the solution of the incompressible Euler equation, we will restrict ourseleves to the discussion of the following results. In \cite{M}, Majda proved  that for $v^{0}\in H^{s}$ with $s>\frac{5}{2},$ the solutions $(v_{\nu})_{\nu>0}$ converge in $L^2$ norm to the unique solution $v$ of the Euler system and the rate convergence is of order $\nu t.$ By using elementary interpolation argument we deduce the strong convergence in the Sobolev spaces $H^{\eta}, \forall \eta<s$. We note that this result is local in time in space dimension $3$ and global in space dimension $2$.  Recently, Masmoudi  proved in \cite{nm06} the strong convergence in the same space of the initial data $H^s$ and his proof is based on the use of a cut-off procedure. We mention that the inviscid limit problem in the context of axisymmetric flows was studied in \cite{hz09}.

Let us now move to the stratified Navier-Stokes system \eqref{a} which has been intensively studied in the last decades and many results were dedicated to the global well-posedness problem. 

The case of stratified Euler equations with axisymmetric initial data was initiated by Hmidi and Rousset in  \cite{tf010}. This system is described by
\begin{equation}\label{fd1} 
\left\{ \begin{array}{ll} 
\partial_{t}v+v\cdot\nabla v+\nabla p=\rho\,e_{z}\\ 
\partial_{t}\rho+v\cdot\nabla\rho-\Delta\rho=0\\
\textnormal{div}\,v=0\\
v_{| t=0}=v^{0}, \quad \rho_{| t=0}=\rho^{0}.  
\end{array} \right.
\end{equation}
To achieve the global existence program one needs to get new a priori estimates especially for  the function $\zeta:=\frac{\omega_\theta}{r}$ which solves the equation,
$$\partial_t\zeta+v\cdot\nabla\zeta=-\frac{\partial_r\rho}{r}.$$
We observe that this equation induces a loss of two derivatives on the density which is exactly what we expect to win from the transport-diffusion equation. Thus the coupling in the  model \eqref{fd1} is critical and the issue of the global existence requires more refined analysis. In \cite{tf010}, the authors gave a positive answer under the assumptions 
$$v^{0}\in H^{s},\quad \rho^{0}\in H^{s-2}\cap L^{m},\, s>{5}/{2},\, m>6\quad\text{and }\quad r^{2}\rho^{0}\in L^2.$$ 
Their basic idea consists in using the coupled function $\Gamma:=\zeta+\frac{\partial_r}{r}\Delta^{-1}\rho$ which satisfies the transport equation 
$$\partial_{t}\Gamma+v\cdot\nabla\Gamma=-\Big[\frac{\partial_r}{r}\Delta^{-1},v\cdot\nabla\Big]\rho.$$
Since the operator $\frac{\partial_r}{r}\Delta^{-1}$  behaves like Riesz transform on the class of axisymmetric functions, the estimate of $\|\zeta(t)\|_{L^{3,1}}$ is equivalent to bound $\|\Gamma(t)\|_{L^{3,1}}$. Therefore the difficulty reduces to the estimate of the singular  commutator which raises in the equation of $\Gamma.$ For this purpose the authors used intensively the axisymmetric structure of the velocity combined with some tools of harmonic analysis and paradifferential calculus. The result of \cite{tf010} was extended in \cite{sul011} by the author of this paper to the framework of critical Besov spaces. More precisely, the global existence  was proved for 
$$ v^{0}\in B_{2,1}^{\frac52},\quad \rho^{0}\in  B_{2,1}^{\frac12}\cap L^{m},\, m>6\quad\text{and }\quad r^{2}\rho^{0}\in L^2.$$

The  aim of this paper is twofold. Firstly, we extend the result of \cite{tf010} to the stratified  Navier-Stokes system \eqref{a} with uniform bounds with respect to the viscosity but for the subcritical regularities, that is, $(v^0,\rho^0)\in H^{s}\times H^{s-2}$ with $s>\frac 52.$ Secondly, we analyze the inviscid limit problem and   we show the strong convergence of the solutions $(v_\nu,\rho_\nu)$ of the system \eqref{a} to the one of \eqref{fd1} in the same space of initial data. We point out that our approach for the last point is completely  different from one's of Masmoudi work \cite{nm06} for the incompressible Navier-Stokes.\\

First of all, we introduce the following space:
$$u\in\chi^{s}_{m}\iff u\in H^{s-2}\cap L^{m}\;\textnormal{and such that}\;\;r^{2}u\in L^2.$$
We state now our main result. 
\begin{Theo}\label{theo1} Let $s>\frac{5}{2},$ $v^{0}\in H^{s}$ be an axisymmetric divergence-free vector field without swirl and $\rho^{0}\in \chi^{s}_{m}$ with $m>6$ an axisymmetric function. 
Then there exists a unique global solution $(v_{\nu},\rho_{\nu})$ for the system \eqref{a} such that,
\begin{equation*}
v_{\nu}\in\mathcal{C}(\RR_+; H^{s})\;\;\textnormal{and}\;\rho_{\nu}\in \mathcal{C}(\RR_+; \chi_{m}^{s})\cap L^{1}_{loc}(\RR_+; Lip),
\end{equation*}
with uniform bounds with respect to the viscosity.\\
Moreover, for any $T>0$ we have
$$
\lim_{\nu\to0}\|(v_\nu-v,\rho_\nu-\rho)\|_{L^\infty_T (H^s\times H^{s-2})}=0,
$$ 
where  $(v,\rho)$ is the solution of  the system \eqref{fd1} associated to the initial data $(v^0,\rho^0).$
\end{Theo}
Before giving some details about the proof few remarks are in order.
\begin{remas}
$(1)$ From the proof the rate convergence in $L^2$ space is of order $\nu t$. More precisely,
$$
\|(v_\nu-v,\rho_\nu-\rho)(t)\|_{ L^2}\le \nu t\,f(t),
$$
with $f$ is an explicit function depending only on the size of the initial data and the variable time $t.$\\
$(2)$ Our approach can not allow to treat the critical case $v_0\in B_{2,1}^\frac52, \rho_0\in B_{2,1}^{\frac12}.$ Even though, we can extend the result of the Proposition \ref{prop jh} to the Lorentz space $L^{3,1},$ the difficulty relies on the establishment of maximal smoothing effects for a transport-diffusion model in Lorentz space.
\end{remas}
Now, we will discuss the main ideas of the proof of Theorem \ref{theo1} and to simplify the notation, we will use $(v,\rho)$ instead of $(v_\nu,\rho_\nu).$ First, recall that the vorticity $\omega=\omega^\theta e_\theta$ satisfies 
$$\partial_{t}\omega+v\cdot\nabla\omega-\nu\Delta\omega=\frac{v^{r}}{r}\omega+curl(\rho e_z).$$
This yields
$$\partial_{t}\omega_{\theta}+v\cdot\nabla\omega_{\theta}-\nu(\Delta\omega_{\theta}-\frac{\omega_{\theta}}{r^2})=\frac{v^{r}}{r}\omega_{\theta}-\partial_{r}\rho.$$
It follows that $\zeta:=\frac{\omega_\theta}{r}$ obeys to the equation,
\begin{equation}\label{b1}
\partial_{t}\zeta+v\cdot\nabla\zeta-\nu(\Delta+\frac{2}{r}\partial_{r})\zeta=-\frac{\partial_{r}\rho}{r}.
\end{equation}
At this stage, we can try to use the method of \cite{tf010} but unfortunately  it  seems to be rigid and fails four the viscous case. Our alternative approach relies on the use of the maximal smoothing effects combined with a suitable commutator estimate. To be more precise, we use 
 interpolation argument combined with the maximum principle leading  for $\bar{p}>3$ to
 
\begin{eqnarray*}
\Vert\zeta(t)\Vert_{L^{3,1}}&\le& C\Vert\zeta(t)\Vert_{L^{2}\cap L^{\bar{p}}}\\
&\le& C\Vert\zeta^{0}\Vert_{L^{2}\cap L^{\bar{p}}}+C\int_{0}^{t}\Big\Vert\frac{\partial_{r}\rho}{r}(\tau)\Big\Vert_{L^{2}\cap L^{\bar{p}}}d\tau.
\end{eqnarray*}
As we will see the restriction of the operator $\frac{\partial_r}{r}$ on the class of axisymmetric functions is dominated by the second derivative:
$$\Vert\frac{\partial_r}{r}\rho(t)\Vert_{L^p}\le C\Vert\nabla^{2}\rho(t)\Vert_{L^p}.$$ 
To estimate this latter quantity we use the maximal smoothing effect of the heatsemi flow and the difficulty reduces to the analysis of the commutator 
$\sum_{j}\|[\Delta_{j},v\cdot\nabla]\rho\|_{L^p}$ which is the hard technical part of this paper. We shall prove in Proposition \ref{prop jh} that  for $p\in]1,+\infty[$ 
\begin{equation}\label{k}
\displaystyle\sum_{j\ge -1}\Big\Vert[\Delta_{j}, v\cdot\nabla]\rho\Big\Vert_{L^p}\lesssim\Vert v \Vert_{L^2}\Vert\rho\Vert_{L^p}+\Vert\frac{\omega_{\theta}}{r}\Vert_{L^{3,1}\cap L^{p}}\bigg(\Vert x_{h}\rho\Vert_{B^{0}_{\infty,1}}+\Vert\rho\Vert_{B^{0}_{p,1}\cap L^{\infty}}\bigg),
\end{equation}
with the notation $x_{h}:=(x_1,x_2).$ 
Consequently, we obtain 
\begin{equation}\label{az}
\Vert\zeta(t)\Vert_{L^{2}\cap L^{\bar{p}}}\le C(t)e^{C\Vert x_{h}\rho\Vert_{L_{t}^{1}B^{0}_{\infty,1}}}.
\end{equation}
To estimate  $\Vert x_{h}\rho\Vert_{L_{t}^{1}B^{0}_{\infty,1}},$ we use the following inequality proved in \cite{tf010},  
$$\Vert x_{h}\rho\Vert_{L_{t}^{1}B^{0}_{\infty,1}}\le C_{0}(t)\Big(1+\int_{0}^{t}h(\tau)\log\big(2+\Vert\zeta\Vert_{L_{\tau}^{\infty}L^{3,1}}\big)d\tau\Big),$$
where $t\mapsto C_{0}(t)$ is a  given continuous  function and $t\mapsto h(t)$ belongs to $L_{loc}^{1}(\RR_+).$  Hence we  conclude by using \eqref{k} and \eqref{az} combined with  Gronwall inequality leading to a global bound for $\|\zeta(t)\|_{L^{3,1}},$ uniformly with respect to the viscosity.

Concerning the inviscid limit, we prove first the strong convergence in $L_{loc}^{\infty}(\RR_+;L^2)$ by performing energy estimates.  
However the strong convergence in the space of the initial data $H^{s}\times H^{s-2}$ is more subtle. We use for this purpose some interpolation arguments combined with  an additional frequency decay of the energy uniformly with respect to $t$ and $\nu$ in the spirit of \cite{t012,stcomp}.

This paper is organized as follows: In section \ref{A2}, we fix some notation, give the definition of Besov and Lorentz spaces and state some smoothing effects for a transport-diffusion equation. In section \ref{A3}, we 
study the estimate of the commutator $\sum_{j\ge-1}[\Delta_{j},v\cdot\nabla]\rho$ in $L^p$ spaces. In the last section, we give the proof of Theorem \ref{theo1} which will be done in several steps.

\section{Tools and functional spaces}\label{A2}
In this preliminary section, we introduce some basic notations and recall the definitions of usual and heterogeneous Besov spaces. We give also some results about Lorentz spaces and discuss some well-known results about the Littlewood-Paley decomposition and a transport-diffusion equation used later. 
\subsection{Notation}
$\bullet$ For any positive $A$ and $B$, the notation  $A\lesssim B$ means that there exists a positive constant $C$ independent of $A$ and $B$ and such that $A\leqslant CB$.\\
$\bullet$ For any pair of operators $X$ and $Y$ acting on some Banach space $\mathcal{A}$, the commutator $[X,Y]$ is defined by $XY-YX.$\\
$\bullet$ For $l\in\NN$, we set
$$\Phi_{l}(t)=C_{0}\underbrace{\exp\Big(...\exp}_{l-times}(C_{0} t^{\frac{19}{6}})...\Big),$$ where $C_{0}$ depends on the norms of the initial data and its value may vary from line to line up to some absolute constants, but it does not depend on the viscosity $\nu.$ We will make an intensive use of the following trivial facts
\begin{equation*}
\int^{t}_{0}\Phi_{l}(\tau) d\tau \le \Phi_{l}(t)\qquad\textnormal{and}\qquad \exp\Big(\int^{t}_{0}\Phi_{l}(\tau) d\tau\Big)\le \Phi_{l+1}(t).
\end{equation*}
To define Besov spaces we need the following dyadic unity partition (see \cite{che98,gl02}). 
\begin{prop}\label{prop15} There exists two nonnegative radial functions $\chi\in\mathscr{D}(\mathbb{R}^{3})$ and $\varphi\in\mathscr{D}(\mathbb{R}^{3}\backslash\{0\})$ such that, 
$$\chi(\xi)+\displaystyle \sum_{j\ge 0}\varphi(2^{-j}\xi)=1, \quad\forall \xi\in\mathbb{R}^{3},$$
$$\vert p-j\vert\ge 2\Rightarrow\mbox{supp }{\varphi}(2^{-p}\cdot)\cap\mbox{supp }{\varphi}(2^{-j}\cdot)=\varnothing,$$
$$j\ge 1\Rightarrow \mbox{supp }{\chi}\cap\mbox{supp }{\varphi}(2^{-j}\cdot)=\varnothing.$$
\end{prop}
Let $f\in\mathcal{S}^{\prime}(\RR^3),$ we define the nonhomogeneous Littlewood-Paley operators by
\begin{equation*}
\Delta_{-1}f=\chi(\DD)f,\;\;\forall j\ge 0,\;\;\Delta_{j}f=\varphi(2^{-j}\DD)f\;\;\textnormal{and}\;\;S_{j}f=\displaystyle \sum_{-1\le k\le j-1}\Delta_{k}f.
\end{equation*}
It may be easily checked that $$f=\sum_{j\ge -1}\Delta_{j}f,\;\;\forall f \in \mathcal{S}^{\prime}(\RR^{3}).$$
Moreover, the Littlewood-Paley operators satisfies the property of almost orthogonality: for any $f,g\in\mathcal{S}^{\prime}(\RR^3),$
$$\Delta_{p}\Delta_{j}f=0\qquad \textnormal{if} \qquad \vert p-j \vert \geqslant 2 \qquad$$
$$\Delta_{p}(S_{j-1}f\Delta_{j}g)=0 \qquad \textnormal{if} \qquad  \vert p-j \vert \geqslant 5.$$\\
The following Bernstein inequality will be of constant use in the paper see \cite{che98}. 
\begin{lem}\label{Bernstein} There exists a constant $C>0$ such that for every $j\in\NN$, $k\in\NN$ and for  every function $v$ we have
\begin{eqnarray*}
\sup_{\vert\alpha\vert=k}\Vert\partial^{\alpha}S_{j}v\Vert_{L^{p_{2}}}&\le& C^{k}2^{j\big(k+3\big(\frac{1}{p_{1}}-\frac{1}{p_{2}}\big)\big)}\Vert S_{j}v \Vert_{L^{p_{1}}},\;\;\textnormal{for}\;\; p_{2}\ge p_{1}\ge 1\\
C^{-k}2^{jk}\Vert\Delta_{j}v\Vert_{L^{p_{1}}}&\le&\sup_{\vert\alpha\vert=k}\Vert\partial^{\alpha}\Delta_{j}v \Vert_{L^{p_{1}}}\le C^{k}2^{jk}\Vert\Delta_{j}v\Vert_{L^{p_{1}}}.
\end{eqnarray*}
\end{lem}
From the paradifferential calculus introduce by J.-M. Bony \cite{bo81} the product $uv$ can be formally divided into three parts as follows :
\begin{equation}\label{sq}
fg=T_{f}g+T_{g}f+R(f,g),
\end{equation}
where 
\begin{equation*}
T_{f}g\overset{def}{=}\sum_{j}S_{j-1}f\Delta_{j}g  
\end{equation*}
and 
\begin{equation*}
R(f,g)=\sum_{j}\Delta_{j}f \widetilde{\Delta}_{j}g\;,\;\textnormal{with}\;\;\widetilde{\Delta}_{j}=\Delta_{j-1}+\Delta_{j}+\Delta_{j+1}.
\end{equation*}
\subsection{Usual and heterogeneous Besov spaces}
We recall now the following definition of general Besov spaces.
\begin{defi}\label{def1}
Let $s\in\RR$ and $1 \le p,r \le +\infty.$ The inhomogeneous Besov space $B_{p,r}^s$ is the set of tempered distributions $f$ such that
$$\Vert f\Vert_{B_{p,r}^s}:=\Big(2^{js}\Vert\Delta_{j}f\Vert_{L^p}\Big)_{\ell^r}<+\infty.$$
\end{defi}
The following embeddings are an easy consequence of Bernstein inequalities,
\begin{equation*}
B_{p_1,r_1}^{s}\hookrightarrow B_{p_2,r_2}^{s+3(\frac{1}{p_2}-\frac{1}{p_1})}\,,\;\,p_1\le p_2\;\;\textnormal{and}\;\; r_1 \le r_2.
\end{equation*}
Let $T>0,$ $\rho\ge 1,$ $(p,r)\in[1,\infty]^{2}$ and $s\in\RR,$ we denote by $L_{T}^{\rho}B_{p,r}^{s}$ the space of distribution $f$ such that
$$\Vert f\Vert_{L_{T}^{\rho}B_{p,r}^{s}}:=\Big\Vert \Big(2^{js}\Vert\Delta_{j}f\Vert_{L^p}\Big)_{\ell^r}\Big\Vert_{L_{T}^{\rho}}<+\infty.$$
We say that $f$ belongs to the Chemin-Lerner space $\widetilde{L}_{T}^{\rho}B_{p,r}^{s}$ if 
$$\Vert f\Vert_{\widetilde{L}_{T}^{\rho}B_{p,r}^{s}}:=\Big\Vert 2^{js}\Vert\Delta_{j}f\Vert_{L_{T}^{\rho}L^p}\Big\Vert_{\ell^r}<+\infty.$$
The relation between these spaces are detailed in the following lemma, which is a direct consequence of the  Minkowski inequality. 
\begin{lem}\label{le1}
 Let $ s\in\RR ,\varepsilon>0$ and $(p,r,\rho) \in[1,+\infty]^3.$ Then we have the following embeddings
$$L^{\rho}_{T}B^{s}_{p,r}\hookrightarrow\widetilde L^{\rho}_{T}B^{s}_{p,r}\hookrightarrow L^{\rho}_{T}B^{s-\varepsilon}_{p,r}\;\;\;\textnormal{if}\quad r\geqslant\rho.$$ 
$${L^\rho_{T}}{B_{p,r}^{s+\varepsilon}}\hookrightarrow\widetilde L^\rho_{T}{B_{p,r}^s}\hookrightarrow L^\rho_{T}B_{p,r}^s\;\;\;\textnormal{if}\quad 
\rho\geq r.$$
\end{lem}
We remark that the Sobolev space $H^s$ coincides with the Besov spaces $B^{s}_{2,2}$ for $s\in\RR$ and we have the following embedding 
$$\forall\,0 \le s <\frac{d}{2}\;\;,\,\,H^{s}\hookrightarrow L^{p}\,\,\textnormal{with}\,\,\;\;\, p=\frac{2d}{d-2s}.$$
Now, we will introduce the heterogeneous Besov spaces which are an extension of the classical Besov spaces.
\begin{defi}\label{de1} Let $\Psi :\{-1\}\cup\NN\to\RR_+^*$ be a given function.\\
$(i)$We say that $\Psi$ belongs to the class $\mathcal{U}$ if the following conditions are satisfied:\\
$(a)$ $\Psi$ is a nondecreasing function.\\
$(b)$ There exists $C>0$ such that,
$$\sup_{x\in\NN\cup\{-1\}}\frac{\Psi(x+1)}{\Psi(x)}\le C.$$
$(ii)$ We define the class $\mathcal{U}_{\infty}$ by the set of function $\Psi\in\mathcal{U}$ satisfying $\displaystyle\lim_{x\to+\infty}\Psi(x)=+\infty.$\\
$(iii)$ Let $s\in\RR,$ $(p,r)\in[1,\infty]^2$ and $\Psi\in\mathcal{U}.$ We define the heterogeneous Besov spaces $B_{p,r}^{s,\Psi}$ as follows:
$$u\in B_{p,r}^{s,\Psi}\;\;\textnormal{iff}\;\;\;\Vert u\Vert_{B_{p,r}^{s,\Psi}}:=\Big(\Psi(q)2^{qs}\Vert\Delta_{q}u\Vert_{L^{p}}\Big)_{\ell^{r}}<+\infty.$$
\end{defi}
We observe that when the profile $\Psi$ has an exponential growth: $\Psi(q)=2^{\alpha q}, \alpha\in\RR_+,$ then the heterogeneous Besov space $B_{p,r}^{s,\Psi}$ reduces to the classical Besov space $B_{p,r}^{s+\alpha}.$ When the profile $\Psi$ is a nonnegative constant, it is clear that $B_{p,r}^{s,\Psi}=B_{p,r}^{s}.$\\
Now, we will give the following result which describes that any element of a given Besov space is always more regular than the prescribed regularity, (see \cite{t012} for a proof).
\begin{lem}\label{Besov}  
Let $s\in\RR,$ $p\in[1,+\infty],$ $r\in[1,+\infty[$ and $f\in B_{p,r}^{s}.$ Then there exists a function $\Psi\in\mathcal{U}_{\infty}$ such that $f\in B_{p,r}^{s,\Psi}.$
\end{lem}
The following Proposition will be useful later, see \cite{hkr} for a proof.
\begin{prop}\label{pro11} We have the following estimates :\\
$a)$ Let $p\in [1,\infty],$ $f,g$ and $h$ be three functions such that $x\,h \in L^{1},$ $\nabla f \in L^{p}$ and $g \in L^{\infty}.$  Then
$$\Vert h \ast (f\,g)-f(h \ast g) \Vert_{L^{p}}\le \Vert x\,h \Vert_{L^{1}}\Vert\nabla f \Vert_{L^{p}}\Vert g \Vert_{L^\infty}.$$
$b)$ Assume that $xh \in L^1$, $\nabla f \in L^\infty$ and $g\in L^{p},$ $\forall p\in [1,\infty].$ Then we have also
$$\Vert h \ast (f\,g)-f(h \ast g) \Vert_{L^{p}}\le \Vert x\,h \Vert_{L^1}\Vert\nabla f \Vert_{L^\infty}\Vert g \Vert_{L^{p}}.$$
\end{prop}
\subsection{Lorentz spaces and interpolation}
We can define the Lorentz spaces by interpolation from Lebesgue spaces :
$$(L^{p_1},L^{p_2})_{(\mu,r)}=L^{p,r},$$
where $1 \le p_{1}< p < p_{2} \le \infty$, $\mu$ satisfies $\frac{1}{p}=\frac{1-\mu}{p_1}+\frac{\mu}{p_2}$ and $1\le r\le\infty.$\\
We have the classical properties:
\begin{equation}\label{claspro}
\Vert uv\Vert_{L^{p,r}}\le C\Vert u\Vert_{L^\infty}\Vert v\Vert_{L^{p,r}}
\end{equation}
\begin{equation*}
L^{p,r}\hookrightarrow L^{p,r_1}\;\;,\;\forall 1 \le p \le \infty,\;1 \le r \le r_1 \le \infty\qquad\textnormal{and}\;\;\;L^{p,p}=L^p.
\end{equation*}
We have also $L^{3,1}=(L^2,L^{\bar{p}})_{(\mu,1)}$ with $3<\bar{p}$ and we deduce that,
\begin{equation}\label{ah}
\Vert u\Vert_{L^{3,1}}\le C\Vert u\Vert_{L^{2}\cap L^{\bar{p}}}\;\;\textnormal{with}\;\,3<\bar{p}.
\end{equation}
The following lemma will be used later see for instance \cite{gl02, ro63}.
\begin{lem}\label{lem h}
There exists a constant $C>0$ such that for every $0<\beta<3$ 
$$\Vert f\ast g \Vert_{L^{\infty}(\RR^3)}\le C \Vert f \Vert_{L^{\frac{3}{\beta},\infty}(\RR^3)}\Vert g \Vert_{L^{\frac{3}{3-\beta},1}(\RR^3)}.$$
\end{lem}
By using this result and the fact that $\frac{1}{\vert x \vert^2}\in L^{\frac{3}{2},\infty}(\RR^3),$ we get
\begin{eqnarray}\label{12}
\nonumber \Vert\nabla\Delta^{-1}f \Vert_{L^{\infty}(\RR^3)}&\lesssim& \Vert\frac{1}{\vert x \vert^2}\Vert_{L^{\frac{3}{2},\infty}(\RR^3)}\Vert f \Vert_{L^{3,1}(\RR^3)}\\
&\lesssim& \Vert f \Vert_{L^{3,1}(\RR^3)}.
\end{eqnarray}

\subsection{Estimates for a transport-diffusion equation}
We will give now some useful estimates for any smooth solution of the linear transport-diffusion model given by
\begin{equation}\label{kp} 
\left\{ \begin{array}{ll} 
\partial_{t}f+v\cdot\nabla f-\kappa\Delta f=g\\
f_{| t=0}=f^{0} 
\end{array} \right.
\end{equation} 
We will give two kinds of estimates : the first is the $L^p$ estimates and the second is the smoothing effects. Let us start with the $L^p$ estimates see \cite{h05}.
\begin{lem}\label{kj1} Let $v$ be a smooth divergence-free vector field of $\RR^3$ and $f$ be a smooth solution of \eqref{kp}. Then we have $\forall p \in [1,\infty]$ and for every $\kappa \ge 0,$  
$$\Vert f(t)\Vert_{L^p}\le \Vert f^{0}\Vert_{L^p}+\int_{0}^{t}\Vert g(\tau)\Vert_{L^p}d\tau.$$ 
\end{lem} 
We need to the following result, see \cite{hk1} for a proof.
\begin{prop}\label{smooth} Let $v$ be a smooth divergence-free vector field of $\RR^3$ with vorticity $\omega:=curl v.$ Let $f$ be a smooth solution of \eqref{kp} with $\kappa=1$ and $g=0$. Then we have for every $j\in\NN,$ $f^{0}\in L^p$ with $1\le p\le\infty$ and $t\ge 0,$  
$$2^{2j}\Vert\Delta_{j} f\Vert_{L_{t}^{1}L^{p}}\lesssim\Vert f^0 \Vert_{L^{p}}\Big(1+(j+1)\Vert\omega\Vert_{L_{t}^{1}L^{\infty}}+\Vert\nabla\Delta_{-1}v \Vert_{L_{t}^{1}L^\infty}\Big).$$
\end{prop}
We will need to the following smoothing effects which are proved in \cite{tf010}.
\begin{prop}\label{pro xc} Let $v$ be a smooth divergence-free vector field of $\RR^3$ and $f$ be a smooth solution of \eqref{kp} with $\kappa=1$. Then we have for every $j\in\NN,$ $p \ge 2$ and $t\ge 0,$
$$\Vert\Delta_{j} f\Vert_{L_{t}^{\infty}L^{p}}+2^{2j}\displaystyle \int_{0}^{t}\Vert\Delta_{j} f(\tau)\Vert_{L^{p}}d\tau \lesssim \Vert\Delta_{j} f^0 \Vert_{L^{p}}+\displaystyle \int_0^t \Vert[\Delta_{j},v \cdot\nabla]f(\tau)\Vert_{L^{p}}d\tau+\int_0^t \Vert\Delta_{j}g(\tau)\Vert_{L^p}d\tau.$$
\end{prop}

\section{Commutator estimates}\label{A3}
In this section, we discuss the commutator between the bloc dyadic $\Delta_j$ and the convection operator $v\cdot\nabla.$ 
First, we start with the following estimate which whose proof can be found in \cite{t012}.
\begin{prop}\label{pr4} 
Let $v$ be a smooth divergence-free vector field of $\RR^3$ and $u$ be a smooth function. Then for every $s>0,$ $r\in[1,+\infty]$ and $\Psi\in\mathcal{U}$ given in Definition \ref{de1}, we have the following estimate
$$\Big(\Psi(j)2^{js}\Vert[\Delta_{j},v\cdot\nabla]u\Vert_{L^2}\Big)_{\ell^r}\lesssim\Vert\nabla v\Vert_{L^\infty}\Vert u\Vert_{B^{s,\Psi}_{2,r}}+\Vert\nabla u\Vert_{L^\infty}\Vert v\Vert_{B^{s,\Psi}_{2,r}}.$$
\end{prop}
The main result of this section is to prove the following,
\begin{prop}\label{prop jh} Let $v$ be an axisymmetric smooth and divergence-free vector field without swirl and $\rho$ be an axisymmetric smooth scalar function. Then for every $j \ge -1$ and $1<p<\infty$, we have the following estimate,
$$\displaystyle\sum_{j\ge -1}\Big\Vert[\Delta_{j}, v\cdot\nabla]\rho\Big\Vert_{L^p}\lesssim\Vert v \Vert_{L^2}\Vert\rho\Vert_{L^p}+\Vert\frac{\omega_{\theta}}{r}\Vert_{L^{3,1}\cap L^{p}}\bigg(\Vert x_{h}\rho\Vert_{B^{0}_{\infty,1}}+\Vert\rho\Vert_{B_{p,1}^{0}\cap L^{\infty}}\bigg),$$
 where $\omega_\theta$ is the angular component of $\omega=\nabla\times v.$
\end{prop}
\begin{proof} We first write by using the decomposition of Bony \eqref{sq} that,
\begin{eqnarray*}
\sum_{j\ge-1} \Big[\Delta_{j}, v\cdot\nabla\Big]\rho
&=& \sum_{j\ge-1}\sum_{i=1}^{3} \Big[\Delta_{j}, T_{v^{i}}\cdot\Big]\partial_{i}\rho+\sum_{j\ge-1}\sum_{i=1}^{3}\Big[\Delta_{j}, T_{\partial_{i}\cdot}\cdot v^{i}\Big]\rho\\
&+& \sum_{j\ge-1}\sum_{i=1}^{3}\Big[\Delta_{j}, R(v^{i}\cdot, \partial_{i})\Big]\rho\\
&:=& \textnormal{I}+\textnormal{II}+\textnormal{III}.
\end{eqnarray*}
\underline{\textbf{Estimate of \textnormal{I}}} :\\
We start with the estimate of the first component of  $\textnormal{I}$ that is for $i=1.$ Since $v$ is a divergence-free, we have $\Delta v=-\nabla\times\omega.$ Then for axisymmetric flows, we obtain that
\begin{eqnarray}\label{br}
\nonumber v^{1}(x)= \Delta^{-1}\partial_{3}\omega^2 &=& \Delta^{-1}\partial_{3}(x_{1}\frac{\omega_\theta}{r})\\
\nonumber &=& \Delta^{-1}(x_{1}\partial_{3}\frac{\omega_\theta}{r})\\
&=& x_{1}\Delta^{-1}\partial_{3}(\frac{\omega_\theta}{r})-2\partial_{13}\Delta^{-2}(\frac{\omega_\theta}{r}).
\end{eqnarray}
In the last line, we have used the following identity, see Lemma 2.10 in \cite{tf010} for a proof.
\begin{lem}\label{lem g}
For every $f\in\mathcal{S}(\RR^3,\RR)$ and $i,j \in\{1,2,3\},$ we have 
$$\Delta^{-1}(x_{i}\partial_{j}f)=x_{i}\Delta^{-1}\partial_{j}f-2\mathcal{R}_{i,j}\Delta^{-1}f,$$
where $\mathcal{R}_{i,j}=\partial_{i,j}\Delta^{-1}$ is the Riesz transform.
\end{lem}
Then we have
\begin{eqnarray}\label{bi}
\nonumber \sum_{j}\Big[\Delta_{j},T_{v^{1}}\cdot\Big]\partial_{1}\rho&=& \sum_{\vert q-j\vert\le 4}\Big[\Delta_{j},S_{q-1} v^{1}\Big]\Delta_{q}\partial_{1}\rho\\
\nonumber &=& \sum_{\vert q-j\vert\le 4}\Big[\Delta_{j},S_{q-1}\big(x_{1}\Delta^{-1}\partial_{3}(\frac{\omega_\theta}{r})\big)\Big]\Delta_{q}\partial_{1}\rho\\
&-& 2\sum_{\vert q-j \vert\le 4}\Big[\Delta_{j},S_{q-1}\partial_{13}\Delta^{-2}(\frac{\omega_\theta}{r})\Big]\Delta_{q}\partial_{1}\rho.
\end{eqnarray}
By the definition of $\Delta_q$, there exists a function $\varphi\in\mathcal{S}(\RR^{3})$ such that,
\begin{eqnarray*}\label{ci}
x_{1}\Delta_{q}\rho&=& x_{1} 2^{3q} \int_{\RR^3}\varphi(2^{q}(x-y))\rho(y) dy\\
&=& 2^{3q} \int_{\RR^3}\varphi(2^{q}(x-y))y_{1}\rho(y) dy+2^{3q} \int_{\RR^3}\varphi(2^{q}(x-y))(x_1-y_1)\rho(y) dy\\
&=& \Delta_{q}(x_{1}\rho)+2^{-q}2^{3q} \varphi_{1}(2^q\cdot)\ast\rho,
\end{eqnarray*}
where $\varphi_1(x)= x_1 \varphi(x).$ Consequently the commutator reads,
\begin{equation}\label{di}
[\Delta_q, x_1]\rho=-2^{2q}\varphi_{1}(2^q\cdot)\ast\rho.
\end{equation}
Similarly for the cutt-off $S_q$, we obtain
\begin{equation}\label{ei}
x_{1}S_{q}F=S_{q}(x_{1}F)+2^{2q}\chi_{1}(2^q\cdot)\ast F,
\end{equation}
where $\chi_{1}(x)= x_{1} \chi(x) \in\mathcal{S}(\RR^3).$\\
To estimate the first term of \eqref{bi}, we use \eqref{ei} to write 
\begin{eqnarray*}
\Big[\Delta_{j}, S_{q-1}\big(x_{1}\Delta^{-1}\partial_{3}(\frac{\omega_\theta}{r})\big)\Big]\Delta_{q}\partial_{1}\rho&=& \Big[\Delta_{j}, x_{1}S_{q-1}\big(\Delta^{-1}\partial_{3}(\frac{\omega_\theta}{r})\big)\Big]\Delta_{q}\partial_{1}\rho\\
&-& \Big[\Delta_{j}, 2^{2q} \chi_{1}(2^q\cdot)\ast\Delta^{-1}\partial_{3}(\frac{\omega_\theta}{r})\Big]\Delta_{q}\partial_{1}\rho\\
&=& \Big[\Delta_{j}, S_{q-1}\big(\Delta^{-1}\partial_{3}(\frac{\omega_\theta}{r})\big)\Big] x_{1}\Delta_{q}\partial_{1}\rho\\
&+& S_{q-1}\Delta^{-1}\partial_{3}(\frac{\omega_\theta}{r})\Big[\Delta_{j},x_{1}\Big]\Delta_{q}\partial_{1}\rho\\
&-&\Big[\Delta_{j}, 2^{2q} \chi_{1}(2^q\cdot)\ast\Delta^{-1}\partial_{3}(\frac{\omega_\theta}{r})\Big]\Delta_{q}\partial_{1}\rho\\
\end{eqnarray*}
Therefore, we obtain
\begin{eqnarray*}
\Big[\Delta_{j}, S_{q-1}\big(x_{1}\Delta^{-1}\partial_{3}(\frac{\omega_\theta}{r})\big)\Big]\Delta_{q}\partial_{1}\rho
&=& \Big[\Delta_{j}, S_{q-1}\big(\Delta^{-1}\partial_{3}(\frac{\omega_\theta}{r})\big)\Big]\partial_{1}(x_{1}\Delta_{q}\rho)\\
&-& \Big[\Delta_{j}, S_{q-1}\big(\Delta^{-1}\partial_{3}(\frac{\omega_\theta}{r})\big)\Big]\Delta_{q}\rho\\
&+& S_{q-1}\Delta^{-1}\partial_{3}(\frac{\omega_\theta}{r})\Big[\Delta_{j},x_{1}\Big]\Delta_{q}\partial_{1}\rho\\
&-&\Big[\Delta_{j}, 2^{2q} \chi_{1}(2^q\cdot)\ast\Delta^{-1}\partial_{3}(\frac{\omega_\theta}{r})\Big]\Delta_{q}\partial_{1}\rho.
\end{eqnarray*}
This gives by using \eqref{di},
\begin{eqnarray*}
\Big[\Delta_{j}, S_{q-1}\big(x_{1}\Delta^{-1}\partial_{3}(\frac{\omega_\theta}{r})\big)\Big]\Delta_{q}\partial_{1}\rho
&=& \Big[\Delta_{j}, S_{q-1}\big(\Delta^{-1}\partial_{3}(\frac{\omega_\theta}{r})\big)\Big]\partial_{1}\Delta_{q}(x_{1}\rho)\\
&+& \Big[\Delta_{j}, S_{q-1}\big(\Delta^{-1}\partial_{3}(\frac{\omega_\theta}{r})\big)\Big]\partial_{1}\big(2^{2q} \varphi_{1}(2^q\cdot)\ast\rho\big)\\
&-& \Big[\Delta_{j}, S_{q-1}\big(\Delta^{-1}\partial_{3}(\frac{\omega_\theta}{r})\big)\Big]\Delta_{q}\rho\\
&-& S_{q-1}\Delta^{-1}\partial_{3}(\frac{\omega_\theta}{r})\Big(2^{2j}\varphi_{1}(2^{j}\cdot)\ast\Delta_{q}\partial_{1}\rho\Big)\\
&-& \Big[\Delta_{j}, 2^{2q}\chi_{1}(2^q\cdot)\ast\Delta^{-1}\partial_{3}(\frac{\omega_\theta}{r})\Big]\Delta_{q}\partial_{1}\rho.
\end{eqnarray*}
Therefore,
\begin{eqnarray*}
\sum_{\vert q-j \vert\le 4}\Big[\Delta_{j}, S_{q-1}\big(x_{1}\Delta^{-1}\partial_{3}(\frac{\omega_\theta}{r})\big)\Big]\Delta_{q}\partial_{1}\rho=:=\textnormal{I}_1+\textnormal{I}_2+\textnormal{I}_3 +\textnormal{I}_4+\textnormal{I}_5,
\end{eqnarray*}
where,
\begin{eqnarray*}
\textnormal{I}_{1}&=&\sum_{\vert q-j \vert\le 4}\Big[\Delta_{j}, S_{q-1}\Big(\Delta^{-1}\partial_{3}(\frac{\omega_\theta}{r})\Big)\Big]\partial_{1}\Delta_{q}(x_{1}\rho),\\
\textnormal{I}_{2}&=& \sum_{\vert q-j \vert\le 4}\Big[\Delta_{j}, S_{q-1}\Big(\Delta^{-1}\partial_{3}(\frac{\omega_\theta}{r})\Big)\Big]\Big(2^{3q}(\partial_{1}\varphi_{1})(2^q\cdot)\ast\rho\Big),\\
\textnormal{I}_{3}&=&-\sum_{\vert q-j \vert\le 4}\Big[\Delta_{j}, S_{q-1}\Big(\Delta^{-1}\partial_{3}(\frac{\omega_\theta}{r})\Big)\Big]\Delta_{q}\rho,\\
\textnormal{I}_{4}&=& -\sum_{\vert q-j \vert\le 4} S_{q-1}\Delta^{-1}\partial_{3}(\frac{\omega_\theta}{r})\Big(2^{2j}\varphi_{1}(2^{j}\cdot)\ast\Delta_{q}\partial_{1}\rho\Big)\\
\textnormal{I}_{5}&=&-\sum_{\vert q-j \vert\le 4}\Big[\Delta_{j}, 2^{2q}\chi_{1}(2^q\cdot)\ast\Delta^{-1}\partial_{3}(\frac{\omega_\theta}{r})\Big]\Delta_{q}\partial_{1}\rho.
\end{eqnarray*}
\underline{Estimate of $\textnormal{I}_{1}.$}
We use Proposition \ref{pro11}-a), the continuity of Riesz transform in the $L^{p}$ space and Bernstein inequality,
\begin{eqnarray*}
\Vert\textnormal{I}_{1}\Vert_{L^{p}}&\le& \sum_{\vert q-j \vert\le 4}\Vert x h_{j}\Vert_{L^1}\Vert\nabla S_{q-1}\Delta^{-1}\partial_{3}(\frac{\omega_\theta}{r})\Vert_{L^{p}}\Vert\partial_{1}\Delta_{q}(x_{1}\rho)\Vert_{L^{\infty}}\\
&\lesssim& \sum_{\vert q-j \vert\le 4}2^{-j}\Vert x h \Vert_{L^1}\Vert\nabla\Delta^{-1}\partial_{3}(\frac{\omega_\theta}{r})\Vert_{L^{p}}2^{q}\Vert\Delta_{q}(x_{1}\rho)\Vert_{L^{\infty}}\\
&\lesssim& \Vert\frac{\omega_\theta}{r}\Vert_{L^{p}}\sum_{\vert q-j \vert\le 4}2^{q-j}\Vert\Delta_{q}(x_{1}\rho)\Vert_{L^{\infty}}\\
&\lesssim& \Vert\frac{\omega_\theta}{r}\Vert_{L^{p}}\Vert x_{1}\rho\Vert_{B^{0}_{\infty,1}},
\end{eqnarray*}
where $h_{j}(x)=2^{3j}h(2^{j}x)\in\mathcal{S}(\RR^3).$\\
\underline{Estimate of $\textnormal{I}_{2}.$} 
Using Proposition \ref{pro11}-a), the continuity of Riesz transform on Lebesgue space and the Young inequalities for convolution, we get 
\begin{eqnarray*}
\Vert\textnormal{I}_{2}\Vert_{L^{p}}&\le& \sum_{\vert q-j \vert\le 4}2^{-j}\Vert x h \Vert_{L^1}\Vert\nabla S_{q-1}\Delta^{-1}\partial_{3}(\frac{\omega_\theta}{r})\Vert_{L^{p}}\Vert 2^{3q} (\partial_{1}\varphi_{1})(2^q\cdot)\ast\rho\Vert_{L^{\infty}}\\
&\lesssim& \Vert\nabla\Delta^{-1}\partial_{3}(\frac{\omega_\theta}{r})\Vert_{L^{p}}\sum_{\vert q-j \vert\le 4}2^{-j}2^{3q}\Vert (\partial_{1}\varphi_{1})(2^q\cdot)\Vert_{L^1}\Vert\rho\Vert_{L^{\infty}}\\
&\lesssim& \Vert\frac{\omega_\theta}{r}\Vert_{L^{p}}\sum_{\vert q-j \vert\le 4}2^{q-j}2^{-q}\Vert \partial_{1}\varphi_{1}\Vert_{L^1}\Vert\rho\Vert_{L^{\infty}}\\
&\lesssim& \Vert\frac{\omega_\theta}{r}\Vert_{L^{p}}\Vert\rho\Vert_{L^{\infty}}.
\end{eqnarray*}
\underline{Estimate of $\textnormal{I}_{3}.$} We use Proposition \ref{pro11}-a) and the continuity of Riesz transform in the Lebesgue space
\begin{eqnarray*}
\Vert\textnormal{I}_{3}\Vert_{L^{p}}&\lesssim& \sum_{\vert q-j \vert\le 4}2^{-j}\Vert x h \Vert_{L^1}\Vert\nabla S_{q-1}\Delta^{-1}\partial_{3}(\frac{\omega_\theta}{r})\Vert_{L^{p}}\Vert\Delta_{q}\rho\Vert_{L^\infty}\\
&\lesssim& \Vert\frac{\omega_\theta}{r}\Vert_{L^{p}}\sum_{\vert q-j \vert\le 4}2^{-j}\Vert\Delta_{q}\rho\Vert_{L^\infty}\\
&\lesssim& \Vert\frac{\omega_\theta}{r}\Vert_{L^{p}}\Vert\rho\Vert_{L^\infty}.
\end{eqnarray*}
\underline{Estimate of $\textnormal{I}_{4}.$}
Using now H\"older inequality, the continuity of the operator $S_{q-1}$ in $L^\infty$ spaces, \eqref{12}, Young inequalities for convolution and Bernstein inequality, we get 
\begin{eqnarray*}
\Vert\textnormal{I}_{4}\Vert_{L^{p}}&\le& \sum_{\vert q-j \vert\le 4}\Vert S_{q-1}\Delta^{-1}\partial_{3}(\frac{\omega_\theta}{r})\Vert_{L^{\infty}}\Vert 2^{2\,j}\varphi_{1}(2^{j}\cdot)\ast\Delta_{q}\partial_{1}\rho\Vert_{L^{p}}\\
&\lesssim& \Vert\frac{\omega_\theta}{r}\Vert_{L^{3,1}}\sum_{\vert q-j \vert\le 4}2^{q-j}\Vert\varphi_{1}\Vert_{L^1}\Vert\Delta_{q}\rho\Vert_{L^p}\\
&\lesssim& \Vert\frac{\omega_\theta}{r}\Vert_{L^{3,1}}\Vert\rho\Vert_{B_{p,1}^{0}}.
\end{eqnarray*}
\underline{Estimate of $\textnormal{I}_{5}.$} We use Proposition \ref{pro11}-a), the Young inequality for the convolution, the continuity of Riesz transform in $L^p$ spaces and Bernstein inequality, we get
\begin{eqnarray*}
\Vert\textnormal{I}_{5}\Vert_{L^{p}}&\lesssim& \sum_{\vert q-j \vert\le 4} 2^{-j}\Vert x h \Vert_{L^1}\Vert2^{2q}\chi_{1}(2^{q}\cdot)\ast\nabla\Delta^{-1}\partial_{3}(\frac{\omega_\theta}{r})\Vert_{L^{p}}\Vert\Delta_{q}\partial_{1}\rho\Vert_{L^{\infty}}\\ 
&\lesssim& \sum_{\vert q-j \vert\le 4}2^{-j}2^{2q}\Vert\chi_{1}(2^{q}\cdot)\Vert_{L^1}\Vert\nabla\Delta^{-1}\partial_{3}(\frac{\omega_\theta}{r})\Vert_{L^{p}}2^{q}\Vert\Delta_{q}\rho\Vert_{L^{\infty}}\\
&\lesssim& \Vert\frac{\omega_\theta}{r}\Vert_{L^{p}}\sum_{\vert q-j \vert\le 4}2^{-j}2^{-q}\Vert\chi\Vert_{L^1}2^{q}\Vert\Delta_{q}\rho\Vert_{L^{\infty}}\\
&\lesssim& \Vert\frac{\omega_\theta}{r}\Vert_{L^{p}}\Vert\rho\Vert_{L^{\infty}}.
\end{eqnarray*}
Finally, we obtain
\begin{equation}\label{kl}
\sum_{\vert q-j \vert\le 4}\Big\Vert\Big[\Delta_{j},S_{q-1}\big(x_{1}\Delta^{-1}\partial_{3}(\frac{\omega_\theta}{r})\big)\Big]\Delta_{q}\partial_{1}\rho\Big\Vert_{L^{p}}\lesssim \Vert\frac{\omega_\theta}{r}\Vert_{L^{3,1}\cap L^{p}}\big(\Vert x_{1}\rho\Vert_{B^{0}_{\infty,1}}+\Vert\rho\Vert_{B_{p,1}^{0}\cap L^{\infty}}\big).
\end{equation}
To estimate the second term of \eqref{bi}, we use Proposition \ref{pro11}-b), Bernstein inequality, \eqref{12} and the continuity of the Riesz transform on the Lorentz spaces, we obtain
\begin{eqnarray}\label{km}
\nonumber \sum_{\vert q-j \vert\le 4}\Big\Vert\Big[\Delta_{j},S_{q-1}\partial_{13}\Delta^{-2}(\frac{\omega_\theta}{r})\Big]\Delta_{q}\partial_{1}\rho\Big\Vert_{L^{p}}&\lesssim& \sum_{\vert q-j \vert\le 4}\Vert x h_{j}\Vert_{L^1}\Vert \nabla S_{q-1}\partial_{13}\Delta^{-2}(\frac{\omega_\theta}{r})\Vert_{L^{\infty}}\Vert\Delta_{q}\partial_{1}\rho\Vert_{L^{p}}\\
\nonumber &\lesssim& \sum_{\vert q-j \vert\le 4}2^{q-j}\Vert xh\Vert_{L^1}\Vert S_{q-1}\partial_{13}\Delta^{-1}(\frac{\omega_\theta}{r})\Vert_{L^{3,1}}\Vert\Delta_{q}\rho\Vert_{L^p}\\
&\lesssim&  \Vert\frac{\omega_\theta}{r}\Vert_{L^{3,1}}\Vert\rho\Vert_{B_{p,1}^{0}}.
\end{eqnarray}
Plugging \eqref{kl} and \eqref{km} into \eqref{bi} we get 
\begin{equation*}
\sum_{j}\Big\Vert\Big[\Delta_{j}, T_{v^{1}}\cdot\Big]\partial_{1}\rho\Big\Vert_{L^{p}}\lesssim \Vert\frac{\omega_\theta}{r}\Vert_{L^{3,1}\cap L^{p}}\Big(\Vert x_{1}\rho\Vert_{B^{0}_{\infty,1}} +\Vert\rho\Vert_{B_{p,1}^{0}\cap L^{\infty}}\Big).
\end{equation*}
The term $\sum_{j}[\Delta_{j},T_{v^2}\cdot]\partial_{2}\rho$ can be estimated in the same way as above and we also obtain the estimate,
$$\sum_{j}\Big\Vert\Big[\Delta_{j}, T_{v^{2}}\cdot\Big]\partial_{2}\rho\Big\Vert_{L^{p}}\lesssim \Vert\frac{\omega_\theta}{r}\Vert_{L^{3,1}\cap L^{p}}\Big(\Vert x_{2}\rho\Vert_{B^{0}_{\infty,1}}+\Vert\rho\Vert_{B_{p,1}^{0}\cap L^{\infty}}\Big).$$
The estimate of the term $\sum_{j}[\Delta_{j},T_{v^3}\cdot]\partial_{3}\rho$ will be done as follows: Since we have,
\begin{eqnarray*}
\Delta v^{3}&=& -(\nabla\times\omega)_{3}\\
&=&-(\partial_{r}\omega_\theta+\frac{\omega_\theta}{r})\\
&=&-\Big(r\partial_{r}\big(\frac{\omega_\theta}{r}\big)+2\frac{\omega_\theta}{r}\Big)\\
&=& -\Big(x_{h}\cdot\nabla_{h}(\frac{\omega_\theta}{r})+2\frac{\omega_\theta}{r}\Big).
\end{eqnarray*}
Using Lemma \ref{lem g}, we get
\begin{eqnarray}\label{kn}
\nonumber -v^{3}(x)&=& \Delta^{-1}\Big(x_{h}\cdot\nabla_{h}(\frac{\omega_\theta}{r})\Big)+2\Delta^{-1}(\frac{\omega_\theta}{r})\\
\nonumber &=& x_{h}\cdot\Delta^{-1}\nabla_{h}(\frac{\omega_\theta}{r})-2\sum^{2}_{i=1}\partial_{ii}\Delta^{-2}(\frac{\omega_\theta}{r})+2\Delta^{-1}(\frac{\omega_\theta}{r})\\
&=& x_{h}\cdot\Delta^{-1}\nabla_{h}(\frac{\omega_\theta}{r})+2\partial_{33}\Delta^{-2}(\frac{\omega_\theta}{r}).
\end{eqnarray}
Then we have a decomposition of the commutator under the form,
\begin{eqnarray*}
-\sum_{j}\Big[\Delta_{j},T_{v^3}\cdot\Big]\partial_{3}\rho&=& \sum_{\vert q-j \vert\le 4}\sum^{2}_{k=1}\Big[\Delta_{j},S_{q-1}\Big(x_{k}\Delta^{-1}\partial_{k}(\frac{\omega_\theta}{r})\Big)\Big]\Delta_{q}\partial_{3}\rho\\
&+& 2\sum_{\vert q-j \vert\le 4}\Big[\Delta_{j},S_{q-1}\partial_{33}\Delta^{-2}(\frac{\omega_\theta}{r})\Big]\Delta_{q}\partial_{3}\rho.
\end{eqnarray*}
This identity looks like \eqref{bi} and then by reproducing the same analysis, we get 
\begin{equation*}
\sum_{j}\Big\Vert\Big[\Delta_{j}, T_{v^{3}}\cdot\Big]\partial_{3}\rho\Big\Vert_{L^{p}}\lesssim\Vert\frac{\omega_\theta}{r}\Vert_{L^{3,1}\cap L^{p}}\Big(\sum_{k=1}^{2}\Vert x_{k}\rho\Vert_{B^{0}_{\infty,1}}+\Vert\rho\Vert_{B_{p,1}^{0}\cap L^{\infty}}\Big).
\end{equation*}
\underline{\textbf{Estimate of \textnormal{II}}} :\\
Let us now turn to the estimate of the second term $\textnormal{II}$. We use \eqref{br}, \eqref{di} and \eqref{ei} with the same computations as for the term $\textnormal{I}$ we get 
\begin{eqnarray*}
\sum_{j}\Big[\Delta_{j}, T_{\partial_1}\cdot v^{1}\Big]\rho&=& \sum_{\vert q-j \vert\le 4}\Big[\Delta_{j},\Delta_{q}v^{1}\Big]S_{q-1}\partial_{1}\rho\\
&:=& \textnormal{II}_1+\textnormal{II}_2+\textnormal{II}_3+\textnormal{II}_4+\textnormal{II}_5+\textnormal{II}_6,
\end{eqnarray*} 
where 
\begin{eqnarray*}
\textnormal{II}_{1}&=& \sum_{\vert q-j \vert\le 4}\Big[\Delta_{j},\Delta_{q}\Delta^{-1}\partial_{3}(\frac{\omega_\theta}{r})\Big]   \partial_{1}S_{q-1}(x_{1}\rho)\\
\textnormal{II}_{2}&=& \sum_{\vert q-j \vert\le 4}\Big[\Delta_{j},\Delta_{q}\Delta^{-1}\partial_{3}(\frac{\omega_\theta}{r})\Big]\Big(2^{3q}(\partial_{1}\chi_{1})(2^{q}\cdot)\ast\rho\Big)\\
\textnormal{II}_{3}&=& -\sum_{\vert q-j \vert\le 4}\big[\Delta_{j}, \Delta_{q}\Delta^{-1}\partial_{3}(\frac{\omega_\theta}{r})\Big]S_{q-1}\rho\\
\textnormal{II}_{4}&=& -\sum_{\vert q-j \vert\le 4}\Delta_{q}\Delta^{-1}\partial_{3}(\frac{\omega_\theta}{r}) \Big(2^{2j}\chi_{1}(2^{j}\cdot)\ast S_{q-1}\partial_{1}\rho\Big)\\
\textnormal{II}_{5}&=& -\sum_{\vert q-j \vert\le 4}\Big[\Delta_{j}, 2^{2q}\varphi_{1}(2^q\cdot)\ast\Delta^{-1}\partial_{3}(\frac{\omega_\theta}{r})\Big] S_{q-1}\partial_{1}\rho\\
\textnormal{II}_{6}&=& -2\sum_{\vert q-j \vert\le 4}\Big[\Delta_{j}, \Delta_{q}\partial_{13}\Delta^{-2}(\frac{\omega_\theta}{r})\Big]S_{q-1}\partial_{1}\rho.
\end{eqnarray*}
To estimate $\textnormal{II}_1,$ we do not need to use the structure of the commutator. We will use H\"older and Bernstein inequalities and the following estimate, we have for every $p\in[1,+\infty]$ that
\begin{equation}\label{w}
\Vert\Delta_{q}\Delta^{-1}\partial_{3} f \Vert_{L^p}\lesssim 2^{-q}\Vert\nabla\Delta_{q}\Delta^{-1}\partial_{3} f \Vert_{L^p}\lesssim 2^{-q}\Vert f \Vert_{L^p}\,,\;\;\;\forall q\ge 0.
\end{equation}
Thus we have,
\begin{eqnarray*}
\Vert\textnormal{II}_{1}\Vert_{L^{p}}&\lesssim& \sum_{\vert q-j \vert \le 4}\Vert\Delta_{q}\big(\Delta^{-1}\partial_{3}(\frac{\omega_\theta}{r})\big)\Vert_{L^p}\Vert\partial_{1}S_{q-1}(x_{1}\rho)\Vert_{L^\infty}\\
&\lesssim& \Vert\frac{\omega_\theta}{r}\Vert_{L^p}\sum_{\vert q-j \vert\le 4}2^{-q}\sum_{-1\le k\le q-2}2^{k}\Vert\Delta_{k}(x_{1}\rho)\Vert_{L^\infty}\\
&\lesssim& \Vert\frac{\omega_\theta}{r}\Vert_{L^{p}}\Vert x_{1}\rho\Vert_{B^{0}_{\infty,1}}.
\end{eqnarray*}
The terms $\textnormal{II}_{2}$, $\textnormal{II}_{3}$, $\textnormal{II}_{4}$, $\textnormal{II}_{5}$ and $\textnormal{II}_{6}$ can be estimated in the similar way of $\textnormal{I}_{2}$, $\textnormal{I}_{3}$, $\textnormal{I}_{4}$, $\textnormal{I}_{5}$ and the second term of \eqref{bi}.
Finally, we conclude that
\begin{equation*}
\Vert\textnormal{II}\Vert_{L^{p}}\lesssim\Vert\frac{\omega_\theta}{r}\Vert_{L^{3,1}\cap L^{p}}\Big(\sum_{i=1}^{2}\Vert x_{i}\rho \Vert_{B^{0}_{\infty,1}}+\Vert\rho\Vert_{B_{p,1}^{0}\cap L^{\infty}}\Big).
\end{equation*}
\underline{\textbf{Estimate of \textnormal{III}}} :\\
Let us now move to the remainder term. We separate it into two terms : the high frequency term and the low frequency term, 
\begin{eqnarray}\label{bb}
\nonumber \textnormal{III}&=& \sum_{j}\Big[\Delta_{j}, R(v, \nabla)\Big]\rho\\
\nonumber &=& \sum_{q\ge j-4}\Big[\Delta_{j}, \Delta_{q}v\widetilde{\Delta}_{q}\nabla\Big]\rho\\
\nonumber &=& \sum^{3}_{j=-1}\Big[\Delta_{j}, \Delta_{-1}v\widetilde{\Delta}_{-1}\nabla\Big]\rho+\sum_{q\ge j-4, \atop{q\in\NN}}\Big[\Delta_{j}, \Delta_{q}v\widetilde{\Delta}_{q}\nabla\Big]\rho\\
&=& \textnormal{III}_1+\textnormal{III}_2.
\end{eqnarray}
To treat the first term, we use Proposition \ref{pro11}-b) and Bernstein inequality
\begin{eqnarray}\label{ab}
\nonumber \Vert\textnormal{III}_{1}\Vert_{L^p}&\le& \sum^{3}_{j=-1}\Vert xh{j}\Vert_{L^1}\Vert\nabla\Delta_{-1}v \Vert_{L^\infty}\Vert\widetilde{\Delta}_{-1}\nabla\rho\Vert_{L^p}\\
\nonumber &\lesssim& \sum^{3}_{j=-1}2^{-j}\Vert xh \Vert_{L^1}\Vert\Delta_{-1}v \Vert_{L^2}\Vert\widetilde{\Delta}_{-1}\rho\Vert_{L^p}\\
&\lesssim& \Vert v \Vert_{L^2}\Vert\rho\Vert_{L^p}.
\end{eqnarray}
For the second term $\textnormal{III}_{2},$ we first write the term inside the sum as follows : 
\begin{eqnarray*}
\Big[\Delta_{j},\Delta_{q}v^{i}\widetilde{\Delta}_{q}\partial_{i}\Big]\rho&=& \Delta_{j}\big(\Delta_{q}v^{i}\widetilde{\Delta}_{q}\partial_{i}\rho\big) -\Delta_{q}v^{i}\widetilde{\Delta}_{q}\partial_{i} \Delta_{j}\rho\\
&=& \Delta_{j}\partial_{i}\big(\Delta_{q}v^{i}\widetilde{\Delta}_{q}\rho\big)- \Delta_{j}\big(\Delta_{q}\partial_{i}v^{i}\widetilde{\Delta}_{q}\rho\big)\\
&-&\partial_{i}
\big(\Delta_{q}v^{i}\widetilde{\Delta}_{q}\Delta_{j}\rho\big)+\Delta_{q}\partial_{i}v^{i}\widetilde{\Delta}_{q}\Delta_{j} \rho.
\end{eqnarray*}
Summing over $i=\{1,2,3\}$ and using the incompressibility of the velocity, we get
$$\sum^{3}_{i=1}\Big[\Delta_{j},\Delta_{q}v^{i}\widetilde{\Delta}_{q}\partial_{i}\Big]\rho=\sum^{3}_{i=1}\bigg(\Delta_{j}\partial_{i} \big(\Delta_{q}v^{i}\widetilde{\Delta}_{q}\rho\big)-\partial_{i} \big(\Delta_{q}v^{i}\widetilde{\Delta}_{q}\Delta_{j}\rho\big)\bigg).$$
Since
$$\widetilde{\Delta}_{q}\Delta_{j}\rho=0\;\,\textnormal{if}\;\,\vert q-j\vert\ge 4\;\;\textnormal{and}\;\;
\Delta_{j}(\Delta_{q}v^{i}\widetilde{\Delta}_{q}\rho)=0\;\,\textnormal{if}\;\,j\ge q+4,$$ we obtain
\begin{eqnarray*}
\textnormal{III}_{2}&=&\sum^{3}_{i=1}\bigg(\sum_{q\ge j-4,\atop{q\in\NN}}\Delta_{j}\partial_{i}\big(\Delta_{q}v^{i}\widetilde{\Delta}_{q}\rho\big)-\sum_{\vert q-j \vert\le 3,\atop{q\in\NN}}\partial_{i}\big(\Delta_{q}v^{i}\widetilde{\Delta}_{q}\Delta_{j}\rho\big)\bigg)\\
&=& \sum_{i=1}^{3}\textnormal{III}_{2}^{i}.
\end{eqnarray*}
Let us now estimate $\textnormal{III}_{2}^{1}.$ Firstly, we follows the same decomposition as \eqref{bi},
\begin{eqnarray*}
\textnormal{III}_{2}^{1}&=& \sum_{q\ge j-4,\atop{q\in\NN}}\Delta_{j}\partial_{1}\big(\Delta_{q}v^{1}\widetilde{\Delta}_{q}\rho\big)-\sum_{\vert q-j \vert\le 3,\atop{q\in\NN}}\partial_{1}\big(\Delta_{q}v^{1}\widetilde{\Delta}_{q}\Delta_{j}\rho\big)\\
&=& \textnormal{III}_{21}^{1}+\textnormal{III}_{22}^{1}+\textnormal{III}_{23}^{1}+\textnormal{III}_{24}^{1}, 
\end{eqnarray*} 
where 
\begin{eqnarray*} 
\textnormal{III}_{21}^{1}&=& \sum_{q\ge j-4,\atop{q\in\NN}}\Delta_{j}\partial_{1}\Big(\Delta_{q}\big(x_{1}\Delta^{-1}\partial_{3}(\frac{\omega_\theta}{r})\big)\widetilde{\Delta}_{q}\rho\Big)\\
\textnormal{III}_{22}^{1}&=& -2\sum_{q\ge j-4,\atop{q\in\NN}} \Delta_{j}\partial_{1}\Big(\Delta_{q}\partial_{13}\Delta^{-2}(\frac{\omega_\theta}{r})\widetilde{\Delta}_{q}\rho\Big)\\
\textnormal{III}_{23}^{1}&=& -\sum_{\vert q-j\vert\le 3,\atop{q\in\NN}}\partial_{1}\Big(\Delta_{q}\big(x_{1}\Delta^{-1}\partial_{3}(\frac{\omega_\theta}{r})\big)\widetilde{\Delta}_{q}\Delta_{j}\rho\Big)\\
\textnormal{III}_{24}^{1}&=& 2\sum_{\vert q-j\vert\le 3,\atop{q\in\NN}}\partial_{1}\Big(\Delta_{q}\partial_{13}\Delta^{-2}(\frac{\omega_\theta}{r})\widetilde{\Delta}_{q}\Delta_{j}\rho\Big). 
\end{eqnarray*}
\underline{Estimate of $\textnormal{III}_{21}^{1}.$}
To estimate the first term $\textnormal{III}_{21}^{1}$, we write by using \eqref{di},
\begin{eqnarray*}
\textnormal{III}_{21}^{1}&=& \sum_{q\ge j-4,\atop{q\in\NN}}\partial_{1}\Delta_{j}\Big(\Delta_{q}\big(\Delta^{-1}\partial_{3}(\frac{\omega_\theta}{r})\big)x_{1}\widetilde{\Delta}_{q}\rho\Big)-\sum_{q\ge j-4,\atop{q\in\NN}}\partial_{1}\Delta_{j}\Big(\big(2^{2q}\varphi_{1}(2^{q}\cdot)\ast\Delta^{-1}\partial_{3}(\frac{\omega_\theta}{r})\big)\widetilde{\Delta}_{q}\rho\Big)\\
&:=& \textnormal{III}_{211}^{1}+\textnormal{III}_{212}^{1}+\textnormal{III}_{213}^{1},
\end{eqnarray*}
where
\begin{eqnarray*}
\textnormal{III}_{211}^{1}&=& \sum_{q\ge j-4,\atop{q\in\NN}}\partial_{1}\Delta_{j}\Big(\Delta_{q}\Delta^{-1}\partial_{3}(\frac{\omega_\theta}{r})\widetilde{\Delta}_{q}(x_{1}\rho)\Big)\\
\textnormal{III}_{212}^{1}&=& \sum_{q\ge j-4,\atop{q\in\NN}}\partial_{1}\Delta_{j}\Big(\Delta_{q}\Delta^{-1}\partial_{3}(\frac{\omega_\theta}{r})\big(2^{2q}\varphi_{1}(2^{q}\cdot)\ast\rho\big)\Big)\\
\textnormal{III}_{213}^{1}&=& -\sum_{q\ge j-4,\atop{q\in\NN}}\partial_{1}\Delta_{j}\Big(\big(2^{2q}\varphi_{1}(2^{q}\cdot)\ast\Delta^{-1}\partial_{3}(\frac{\omega_\theta}{r})\big)\widetilde{\Delta}_{q}\rho\Big).       
\end{eqnarray*}
To estimate the term $\textnormal{III}_{211}^{1},$ we use Bernstein and H\"older inequalities and \eqref{w}, we find
\begin{eqnarray*}
\Vert\textnormal{III}_{211}^{1}\Vert_{L^{p}}&\lesssim& \sum_{q\ge j-4,\atop{q\in\NN}} 2^{j}\Vert\Delta_{q}\Delta^{-1}\partial_{3}(\frac{\omega_\theta}{r})\Vert_{L^{p}}\Vert\widetilde{\Delta}_{q}(x_{1}\rho)\Vert_{L^{\infty}}\\
&\lesssim& \sum_{q\ge j-4,\atop{q\in\NN}}2^{j-q}\Vert\nabla\Delta_{q}\Delta^{-1}\partial_{3}(\frac{\omega_\theta}{r})\Vert_{L^{p}}\Vert\widetilde{\Delta}_{q}(x_{1}\rho)\Vert_{L^\infty}\\
&\lesssim& \Vert\frac{\omega_\theta}{r}\Vert_{L^{p}}\sum_{q\ge j-4,\atop{q\in\NN}}2^{j-q}\Vert\widetilde{\Delta}_{q}(x_{1}\rho)\Vert_{L^{\infty}}\\
&\lesssim& \Vert\frac{\omega_\theta}{r}\Vert_{L^{p}}\Vert x_{1}\rho \Vert_{B^{0}_{\infty,1}}.
\end{eqnarray*}
Now to estimate the terms $\textnormal{III}_{212}^{1},$ we use Young inequality and \eqref{w}, we find 
\begin{eqnarray*}
\Vert\textnormal{III}_{212}^{1}\Vert_{L^{p}}&\lesssim& \sum_{q\ge j-4,\atop{q\in\NN}}2^{j}\Vert\Delta_{q}\Delta^{-1}\partial_{3}(\frac{\omega_\theta}{r})\Vert_{L^{p}}\Vert 2^{2q}\varphi_{1}(2^{q}\cdot)\ast\rho\Vert_{L^{\infty}}\\
&\lesssim& \Vert\frac{\omega_\theta}{r}\Vert_{L^{p}}\sum_{q\ge j-4,\atop{q\in\NN}}2^{j-q}2^{-q}\Vert\varphi_{1}\Vert_{L^1}\Vert\rho\Vert_{L^{\infty}}\\
&\lesssim& \Vert\frac{\omega_\theta}{r}\Vert_{L^{p}}\Vert\rho\Vert_{L^{\infty}}.
\end{eqnarray*}
The term $\textnormal{III}_{213}^{1}$ can be estimated by using Bernstein and H\"older inequalities, combined with the convolution inequality and \eqref{12},
\begin{eqnarray*}
\Vert\textnormal{III}_{213}^{1}\Vert_{L^{p}}
&\lesssim& \sum_{q\ge j-4,\atop{q\in\NN}}2^{j}\Vert 2^{2q}\varphi_{1}(2^{q}\cdot)\ast\Delta^{-1}\partial_{3}(\frac{\omega_\theta}{r}) \Vert_{L^{\infty}}\Vert\widetilde{\Delta}_{q}\rho\Vert_{L^{p}}\\
&\lesssim& \sum_{q\ge j-4,\atop{q\in\NN}}2^{j-q}\Vert\varphi_{1}\Vert_{L^1}\Vert\Delta^{-1}\partial_{3}(\frac{\omega_\theta}{r})\Vert_{L^{\infty}}\Vert\widetilde{\Delta}_{q}\rho\Vert_{L^{p}}\\
&\lesssim& \Vert\frac{\omega_\theta}{r}\Vert_{L^{3,1}}\sum_{q\ge j-4,\atop{q\in\NN}}2^{j-q}\Vert\widetilde{\Delta}_{q}\rho\Vert_{L^{p}}\\
&\lesssim& \Vert\frac{\omega_\theta}{r}\Vert_{L^{3,1}}\Vert\rho\Vert_{B_{p,1}^{0}}.
\end{eqnarray*}
Thus, we obtain
\begin{equation}\label{ac}
\Vert\textnormal{III}_{21}^{1}\Vert_{L^{p}}\lesssim\Vert\frac{\omega_\theta}{r}\Vert_{L^{3,1}\cap L^{p}}\Big(\Vert x_{1}\rho\Vert_{B_{\infty,1}^{0}}+\Vert\rho\Vert_{B^{0}_{p,1}\cap L^{\infty}}\Big).
\end{equation}
\underline{Estimate of $\textnormal{III}_{22}^{1}.$} 
Thanks to the Bernstein inequality, we have for every $p\in[1,+\infty]$ that,
\begin{equation}\label{ww}
\Vert\Delta_{q}\partial_{13}\Delta^{-2}f \Vert_{L^p}\lesssim 2^{-2q}\Vert\nabla^{2}\Delta_{q}\partial_{13}\Delta^{-2} f \Vert_{L^p}\lesssim 2^{-2q}\Vert f \Vert_{L^p}\,,\;\;\;\forall q\ge 0.
\end{equation}
This yields to,
\begin{eqnarray}\label{ad}
\nonumber \Vert\textnormal{III}_{22}^{1}\Vert_{L^{p}}&\lesssim& \sum_{q\ge j-4,\atop{q\in\NN}} 2^{j}\Vert\Delta_{q}\partial_{13}\Delta^{-2}(\frac{\omega_\theta}{r})\Vert_{L^p}\Vert\widetilde{\Delta}_{q}\rho\Vert_{L^{\infty}}\\
\nonumber& \lesssim& \Vert\frac{\omega_\theta}{r}\Vert_{L^{p}}\sum_{q\ge j-4,\atop{q\in\NN}}2^{j-q}2^{-q}\Vert\widetilde{\Delta}_{q}\rho\Vert_{L^{\infty}}\\
&\lesssim& \Vert\frac{\omega_\theta}{r}\Vert_{L^{p}}\Vert\rho\Vert_{L^{\infty}}.
\end{eqnarray}
\underline{Estimate of $\textnormal{III}_{23}^{1}.$}
This term can be written in the similar way as $\textnormal{III}_{21}^{1}$ and we get finally:
\begin{eqnarray*}
\textnormal{III}_{23}^{1}= \textnormal{III}_{231}^{1}+\textnormal{III}_{232}^{1}+\textnormal{III}_{233}^{1}+\textnormal{III}_{234}^{1},
\end{eqnarray*}
where
\begin{eqnarray*}
\textnormal{III}_{231}^{1}&=& -\sum_{\vert q-j\vert\le 3,\atop{q\in\NN}}\partial_{1}\Big(\Delta_{q}\Delta^{-1}\partial_{3}(\frac{\omega_\theta}{r})\widetilde{\Delta}_{q}\Delta_{j}\big(x_{1}\rho\big)\Big)\\
\textnormal{III}_{232}^{1}&=& -\sum_{\vert q-j \vert\le 3,\atop{q\in\NN}}\partial_{1}\Big(\Delta_{q}\Delta^{-1}\partial_{3}(\frac{\omega_\theta}{r})\widetilde{\Delta}_{q}\big(2^{2j}\varphi_{1}(2^{j}\cdot)\ast\rho\big)\Big)\\
\textnormal{III}_{233}^{1}&=& -\sum_{\vert q-j \vert\le 3,\atop{q\in\NN}}\partial_{1}\Big(\Delta_{q}\Delta^{-1}\partial_{3}(\frac{\omega_\theta}{r})\big(2^{2q}\varphi_{1}(2^{q}\cdot)\ast\Delta_{j}\rho\big)\Big)\\
\textnormal{III}_{234}^{1}&=& \sum_{\vert q-j \vert\le 3,\atop{q\in\NN}}\partial_{1}\Big(\big(2^{2q}\varphi_{1}(2^{q}\cdot)\ast\Delta^{-1}\partial_{3}(\frac{\omega_\theta}{r})\big)\widetilde{\Delta}_{q}\Delta_{j}\rho\Big).
\end{eqnarray*}
We point out that by reproducing the same analysis as for $\textnormal{III}_{21}^{1},$ we get
\begin{equation}\label{ae}
\Vert\textnormal{III}_{23}^{1}\Vert_{L^{p}}\lesssim \Vert\frac{\omega_\theta}{r}\Vert_{L^{3,1}\cap L^{p}}\Big(\Vert x_{1}\rho \Vert_{B^{0}_{\infty,1}}+\Vert\rho\Vert_{B^{0}_{p,1}\cap L^{\infty}}\Big).
\end{equation}
\underline{Estimate of $\textnormal{III}_{24}^{1}.$} Using \eqref{ww}, we find 
\begin{eqnarray}\label{af}
\nonumber \Vert\textnormal{III}_{24}^{1}\Vert_{L^{p}}&\lesssim& \sum_{\vert q-j \vert\le 3,\atop{q\in\NN}} 2^{j}\Vert\Delta_{q}\partial_{13}\Delta^{-2}(\frac{\omega_\theta}{r})\Vert_{L^{p}}\Vert\widetilde{\Delta}_{q}\Delta_{j}\rho\Vert_{L^{\infty}}\\
\nonumber &\lesssim& \Vert\frac{\omega_\theta}{r}\Vert_{L^{p}}\sum_{\vert q-j \vert\le 3,\atop{q\in\NN}} 2^{j-q}2^{-q}\Vert\widetilde{\Delta}_{q}\rho\Vert_{L^{\infty}}\\
&\lesssim& \Vert\frac{\omega_\theta}{r}\Vert_{L^{p}}\Vert\rho\Vert_{L^{\infty}}.
\end{eqnarray}
Putting together \eqref{ac}, \eqref{ad}, \eqref{ae} and \eqref{af}, we find finally 
\begin{equation}\label{3i}
\Vert\textnormal{III}_{2}^{1}\Vert_{L^p}\lesssim\Vert\frac{\omega_\theta}{r}\Vert_{L^{3,1}\cap L^{p}}\Big(\Vert x_{1}\rho\Vert_{B^{0}_{\infty,1}}+\Vert\rho\Vert_{B^{0}_{p,1}\cap L^{\infty,}}\Big).
\end{equation}
The term $\textnormal{III}_{2}^{2}$ can be estimated as for the estimate of $\textnormal{III}_{2}^{1}.$ For the term  $\textnormal{III}_{2}^{3},$ we use \eqref{kn} and then by reproducing the same analysis, we get with the notation $x_{h}:=(x_1,x_2)$ the estimate
$$\Vert\textnormal{III}_{2}^{2}\Vert_{L^p}+\Vert\textnormal{III}_{2}^{3}\Vert_{L^p} \lesssim\Vert\frac{\omega_\theta}{r}\Vert_{L^{3,1}\cap L^{p}}\Big(\Vert x_{h}\rho\Vert_{B^{0}_{\infty,1}}+\Vert\rho\Vert_{B^{0}_{p,1}\cap L^{\infty,}}\Big).$$
Combining the above estimate with \eqref{3i}, yields
\begin{equation}\label{i3j}
\Vert\textnormal{III}_{2}\Vert_{L^p}\lesssim\Vert\frac{\omega_\theta}{r}\Vert_{L^{3,1}\cap L^{p}}\Big(\Vert x_{h}\rho\Vert_{B^{0}_{\infty,1}}+\Vert\rho\Vert_{B^{0}_{p,1}\cap L^{\infty,}}\Big).
\end{equation}
Now, from \eqref{ab} and \eqref{i3j} we get
\begin{eqnarray*}
\Vert\textnormal{III}\Vert_{L^{p}}\lesssim\Vert v\Vert_{L^2}\Vert\rho\Vert_{L^p}+\Vert\frac{\omega_\theta}{r}\Vert_{L^{3,1}\cap L^{p}}\Big(\Vert x_{h}\rho \Vert_{B^{0}_{\infty,1}}+\Vert\rho\Vert_{B^{0}_{p,1}\cap L^{\infty}}\Big).
\end{eqnarray*}
This ends the proof of the proposition.
\end{proof}

\section{Proof of Theorem \ref{theo1}}\label{A4}
To prove Theorem \ref{theo1}, we will restrict ourself to prove some a priori estimates and the inviscid limit. The proof of the uniqueness and the existence of the solutions are standard. 

\subsection{A priori estimates}
We establish in this subsection some global a priori estimates which we need in the proof of our main result. First we give some energy estimates and we shall prove an estimate of $\Vert\frac{v^r}{r}\Vert_{L^\infty}$ which is based on the estimation of our commutator in the previous section. Finally we will establish a control of the norm Lipschitz of the velocity. Let us start with the energy estimates.
\subsubsection{Energy estimates} 
We have the following estimates
\begin{prop}\label{prop a} Let $(v,\rho)$ be a smooth solution of \eqref{a}, then we have\\
$(a)$ For $(v^{0},\rho^{0})\in L^{2}\times L^{2},$ $t \in \RR_+$ and $\nu\ge 0$ we have
$$\Vert v(t)\Vert^{2}_{L^2}+2\nu\int_{0}^{t}\Vert\nabla v(\tau)\Vert^{2}_{L^2}d\tau\le C_0 (1+t^{2}),$$
where $C_0$ depends only on $\Vert v^{0} \Vert_{L^2}$ and $\Vert\rho^{0}\Vert_{L^2}$ but not on the viscosity $\nu.$\\ 
$(b)$ For $\rho^{0}\in L^{2}$ we have,
\begin{equation*}
\Vert\rho\Vert^{2}_{L_{t}^{\infty}L^{2}}+2\Vert\nabla\rho\Vert^{2}_{L_{t}^{2}L^{2}}=\Vert\rho^0 \Vert^{2}_{L^2}\;\;\textnormal{and}\;\;\Vert\rho(t)\Vert_{L^{\infty}}\le C t^{-\frac{3}{4}}\Vert\rho^{0}\Vert_{L^2}.
\end{equation*}
The constant $C$ does not depend on the viscosity.
\end{prop}
Remark that the axisymmetric assumption on the velocity and the density is not needed in this Proposition.
The proof of the first estimate $(a)$ can be found in \cite{hts09}. For the proof of $(b)$ see \cite{tf010}.\\
We aim now to give some estimates for the horizontal moment $x_{h}\rho$ of the density that will be needed later, 
see Proposition 4.2 (1)-(3) in \cite{tf010} for a proof. 
\begin{prop}\label{prop ay1} Let $v$ be a smooth vector field with zero divergence and $\rho$ be a smooth solution of the second equation of \eqref{a}. Then we have\\
$(1)$ If $\rho^{0}\in L^{2}$ and $x_{h}\rho^{0}\in L^2,$ then there exists $C_{0}>0$ such that for every $t \in \RR_+,$ 
\begin{equation*}
\Vert x_{h}\rho\Vert_{L_{t}^{\infty}L^2}+\Vert x_{h}\rho\Vert_{L_{t}^{2}\dot{H}^{1}}\le C_{0}(1+t^{\frac{5}{4}}).
\end{equation*}
$(2)$ If $\rho^{0}\in L^{2}$ and $\vert x_{h}\vert^{2}\rho^{0}\in L^{2},$ then there exists $C_{0}>0$ such that for every $t \in \RR_+,$ 
$$\Big\Vert\vert x_{h}\vert^{2}\rho\Big\Vert_{L_{t}^{\infty}L^2}+\Big\Vert\vert x_{h}\vert^{2}\rho\Big\Vert_{L_{t}^{2}\dot{H}^{1}}\le C_{0}(1+t^{\frac{5}{2}}).$$
Where $C_0$ depend only on the norm of the initial data and not on the viscosity.
\end{prop}

\subsubsection{Strong estimates}
We will prove in the first step a bound for $\Vert\frac{\omega}{r}(t)\Vert_{L^{3,1}}$ which is the important quantity to get the global existence of smooth solutions. It allows us to bound for all times the vorticity in $L^\infty$ space and then to bound the Lipschitz norm of the velocity $\Vert\nabla v(t)\Vert_{L^\infty}.$
\begin{prop}\label{prop b} Let $v^0$ be a smooth axisymmetric vector field with zero divergence such that $v^0\in L^2$, its vorticity such that $\frac{\omega^0}{r}\in L^{2}\cap L^{\bar{p}}$ with $3<\bar{p}<6$ and let $\rho^{0}\in B_{2,1}^{0}\cap B_{\bar{p},1}^{0}\cap L^{m}$ with $m>6$ be an axisymmetric function, such that $\vert x_{h}\vert^{2}\rho^{0}\in L^{2}.$ Then we have for every $t \in \RR_+$ 
$$\Vert\frac{\omega}{r}(t)\Vert_{L^{3,1}}+\Vert \frac{v^r}{r}(t)\Vert_{L^{\infty}}\le\Phi_{2}(t).$$
We recall that $\Phi_{2}(t)=C_{0}e^{\exp\{C_{0}t^{\frac{19}{6}}\}}$ and the constant $C_0$ depends only on the norm of the initial data but not on the viscosity $\nu.$
\end{prop}
\begin{Rema}\label{rem12}
We note that for $\rho^{0}\in H^{s-2}$ with $s>\frac{5}{2},$ there exists $\bar{p}>3$ such that $\rho^{0}\in B_{2,1}^{0}\cap B_{\bar{p},1}^{0}.$
\end{Rema}
\begin{proof} We start with using the following result, proved in \cite{ahs08}
\begin{eqnarray*}
\vert v^{r}/r \vert\lesssim \frac{1}{\vert\cdot\vert^{2}}\ast\vert\frac{\omega_\theta}{r}\vert.
\end{eqnarray*}
Using Lemma \ref{lem h} and \eqref{ah}, we have for $\bar{p}>3$
\begin{eqnarray}\label{as}
\nonumber \Vert v^{r}/r \Vert_{L^\infty}&\lesssim& \Vert\frac{1}{\vert\cdot\vert^{2}}\Vert_{L^{\frac{3}{2},\infty}}\Vert\frac{\omega_\theta}{r}\Vert_{L^{3,1}}\\
&\lesssim& \Vert\frac{\omega_\theta}{r}\Vert_{L^{3,1}}\lesssim\Vert\frac{\omega_\theta}{r}\Vert_{L^{2}\cap L^{\bar{p}}}.
\end{eqnarray}

It remains then to estimate $\Vert\frac{\omega_\theta}{r}\Vert_{L^{2}\cap L^{\bar{p}}}.$ For this purpose we recall that the function $\zeta:=\frac{\omega_\theta}{r}$ satisfies the equation 
$$\partial_{t}\zeta+v\cdot\nabla\zeta-\nu(\Delta+\frac{2}{r}\partial_r)\zeta=-\frac{\partial_{r}\rho}{r}.$$
By performing $L^p$ estimates, we get
\begin{equation}\label{as0}
\Vert\zeta(t)\Vert_{L^{2}\cap L^{\bar{p}}}\le \Vert\zeta^{0}\Vert_{L^{2}\cap L^{\bar{p}}}+\int_{0}^{t}\Vert\frac{\partial_{r}\rho}{r}(\tau)\Vert_{L^{2}\cap L^{\bar{p}}}d\tau.
\end{equation}
At this stage we need the following lemma which we refer to \cite{tf010}.
\begin{lem}\label{lem1}
For every axisymmetric smooth scalar function $u$, we have
$$\frac{\partial_{r}}{r}u=\sum_{i,j=1}^{2}b_{ij}(x)\partial_{ij}u,$$
where the functions $b_{ij}$ are bounded. 
\end{lem}
Consequently for every $1\le p \le\infty,$ we obtain 
$$\Vert\frac{\partial_{r}}{r}u \Vert_{L^p}\lesssim\Vert\nabla^{2}u \Vert_{L^p}.$$
By using this Lemma, Bernstein inequality, we obtain
\begin{eqnarray*}
\Vert\frac{\partial_{r}\rho}{r}\Vert_{L_{t}^{1}(L^{2}\cap L^{\bar{p}})}
&\lesssim& \Vert\nabla^{2}\rho\Vert_{L_{t}^{1}(L^{2}\cap L^{\bar{p}})}\\
&\lesssim& \sum_{j\ge-1}\Vert\Delta_{j}\nabla^{2}\rho\Vert_{L_{t}^{1}(L^{2}\cap L^{\bar{p}})}\\ 
&\lesssim& \sum_{j\ge-1}2^{2j}\Vert\Delta_{j}\rho\Vert_{L_{t}^{1}(L^{2}\cap L^{\bar{p}})}\\
&\lesssim& \int_{0}^{t}\Vert\Delta_{-1}\rho(\tau)\Vert_{L^{2}\cap L^{\bar{p}}}d\tau+\sum_{j\ge 0}2^{2j}\int_{0}^{t}\Vert\Delta_{j}\rho(\tau)\Vert_{L^{2}\cap L^{\bar{p}}}d\tau.
\end{eqnarray*}
Therefore by using Lemme \ref{kj1}, Proposition \ref{pro xc} and Proposition \ref{prop jh}, we obtain
\begin{eqnarray*}
\Vert\frac{\partial_{r}\rho}{r}\Vert_{L_{t}^{1}(L^{2}\cap L^{\bar{p}})}&\lesssim& t\Vert\rho^{0}\Vert_{L^{2}\cap L^{\bar{p}}}+\sum_{j\ge 0}\Vert\Delta_{j}\rho^{0}\Vert_{L^{2}\cap L^{\bar{p}}}+\sum_{j\ge 0} \Vert[\Delta_{j},v\cdot\nabla]\rho\Vert_{L_{t}^{1}(L^{2}\cap L^{\bar{p}})}\\
&\lesssim& \Vert\rho^{0}\Vert_{B_{2,1}^{0}\cap B_{\bar{p},1}^{0}}(1+t)+\int_{0}^{t}\Vert v(\tau) \Vert_{L^2}\Vert\rho(\tau)\Vert_{L^{2}\cap L^{\bar{p}}}d\tau\\
&+& \int_{0}^{t}\Vert\zeta(\tau)\Vert_{L^{2}\cap L^{\bar{p}}}\big(\Vert x_{h}\rho(\tau)\Vert_{B^{0}_{\infty,1}}+\Vert\rho(\tau)\Vert_{B^{0}_{2,1}\cap B^{0}_{\bar{p},1}\cap L^{\infty}}\big)d\tau\\
&\lesssim& \Vert\rho^{0}\Vert_{B_{2,1}^{0}\cap B_{\bar{p},1}^{0}}(1+t)+\Vert\rho^{0}\Vert_{L^{2}\cap L^{\bar{p}}}\int_{0}^{t}\Vert v(\tau)\Vert_{L^2}d\tau\\
&+& \int_{0}^{t}\Vert\zeta(\tau)\Vert_{L^{2}\cap L^{\bar{p}}}\big(\Vert x_{h}\rho(\tau)\Vert_{B^{0}_{\infty,1}}+\Vert\rho(\tau)\Vert_{B^{0}_{2,1}\cap B^{0}_{\bar{p},1}\cap L^{\infty}}\big)d\tau.
\end{eqnarray*}
Plugging this last estimate in \eqref{as0} and using Proposition \ref{prop a}, we get
\begin{eqnarray*}
\Vert\zeta(t)\Vert_{L^{2}\cap L^{\bar{p}}}&\lesssim& \Vert\zeta^{0}\Vert_{L^{2}\cap L^{\bar{p}}}+\Vert\rho^{0}\Vert_{B_{2,1}^{0}\cap B_{\bar{p},1}^{0}}(1+t)+ \Vert\rho^{0}\Vert_{L^{2}\cap L^{\bar{p}}}C_{0}t(1+t)\\
&+& \int_{0}^{t}\Vert\zeta(\tau)\Vert_{L^{2}\cap L^{\bar{p}}}\big(\Vert x_{h}\rho(\tau)\Vert_{B^{0}_{\infty,1}}+\Vert\rho(\tau)\Vert_{B^{0}_{2,1}\cap B^{0}_{\bar{p},1}\cap L^{\infty}}\big)d\tau.
\end{eqnarray*}
Gronwall's inequality gives
\begin{equation}\label{as00} 
\Vert\zeta(t)\Vert_{L^{2}\cap L^{\bar{p}}}\le C_{0}(1+t^2)\exp\Big\{C\Vert x_{h}\rho\Vert_{L_{t}^{1}B^{0}_{\infty,1}}+C\Vert\rho\Vert_{L_{t}^{1}\big(B^{0}_{2,1}\cap B^{0}_{\bar{p},1}\cap L^{\infty}\big)}\Big\}.
\end{equation} 
To estimate the term $\Vert\rho\Vert_{L_{t}^{1}B^{0}_{2,1}},$ we use the embedding $B^{1/2}_{2,1}\hookrightarrow B^{0}_{2,1},$ interpolation estimate, H\"older inequality and Proposition \ref{prop a},
$$\Vert\rho\Vert_{L_{t}^{1}B^{0}_{2,1}}\lesssim\Vert\rho\Vert_{L_{t}^{1}B^{\frac{1}{2}}_{2,1}}\lesssim\Vert\rho\Vert^{\frac{1}{2}}_{L_{t}^{\infty}L^{2}}t^{\frac{3}{4}}\Vert\nabla\rho\Vert^{\frac{1}{2}}_{L_{t}^{2}L^{2}}\lesssim t^{\frac{3}{4}}\Vert\rho^{0}\Vert_{L^2}.$$
The term $\Vert\rho\Vert_{L_{t}^{1}L^{\infty}},$ can be estimate by using the second estimate of Proposition \ref{prop a}-(b) and integrating in time we get
$$\Vert\rho\Vert_{L_{t}^{1}L^{\infty}}\lesssim t^{\frac{1}{4}}\Vert\rho^{0}\Vert_{L^2}$$
and for $\Vert\rho\Vert_{L_{t}^{1}B_{\bar{p},1}^{0}},$ we have by definition of Besov spaces and for $2<\bar{p}<6$ and Proposition \ref{prop a}-(b) that
\begin{eqnarray*}
\Vert\rho\Vert_{L_{t}^{1}B_{\bar{p},1}^{0}}= \sum_{q\ge-1}\Vert\Delta_{q}\rho\Vert_{L_{t}^{1}L^{\bar{p}}}&\lesssim& \sum_{q\ge-1}2^{3q(\frac{1}{2}-\frac{1}{\bar{p}})}\Vert\Delta_{q}\rho\Vert_{L_{t}^{1}L^{2}}\\
&\lesssim& \sum_{q\ge-1}2^{q(\frac{1}{2}-\frac{3}{\bar{p}})}2^{q}\Vert\Delta_{q}\rho\Vert_{L_{t}^{1}L^{2}}\\
&\lesssim& \Vert\rho\Vert_{L_{t}^{1}H^{1}}\\
&\lesssim& t^{\frac{1}{2}}\Vert\rho\Vert_{L_{t}^{2}H^{1}}\\
&\lesssim& t^{\frac{1}{2}}\Vert\rho^{0}\Vert_{L^{2}}.
\end{eqnarray*}
Consequently we obtain in view of \eqref{as00},
\begin{eqnarray}\label{conss}
\nonumber \Vert\zeta(t)\Vert_{L^{2}\cap L^{\bar{p}}}&\le& C_{0}(1+t^{2})e^{C(t^{\frac{3}{4}}+t^{\frac{1}{2}}+ t^{\frac{1}{4}})\Vert\rho^{0}\Vert_{L^{2}}}e^{C\Vert x_{h}\rho\Vert_{L_{t}^{1}B^{0}_{\infty,1}}}\\
&\le& C_{0}e^{C_{0}t^{2}} e^{C\Vert x_{h}\rho\Vert_{L_{t}^{1}B^{0}_{\infty,1}}}.
\end{eqnarray}
To estimate the term $\Vert x_{h}\rho\Vert_{L_{t}^{1}B^{0}_{\infty,1}},$ we use the following inequality proved   for $\rho^{0}\in L^{2}\cap L^{m}$ with $m>6$ and $\vert x_{h}\vert^{2}\rho^{0}\in L^{2}$ (see \cite{tf010} for a proof),
$$\Vert x_{h}\rho\Vert_{L_{t}^{1}B^{0}_{\infty,1}}\le C_{0}(1+t^{\frac{19}{6}})+C_{0}\int_{0}^{t}(\tau^{\frac{13}{6}}+\tau^{-\frac{3}{4}})\log\big(2+\Vert\zeta\Vert_{L_{\tau}^{\infty}L^{3,1}}\big)d\tau.$$
Hence, for $\bar{p}>3$ we get
\begin{equation}\label{logzeta}
\Vert x_{h}\rho\Vert_{L_{t}^{1}B^{0}_{\infty,1}}\le C_{0}(1+t^{\frac{19}{6}})+C_{0}\int_{0}^{t}(\tau^{\frac{13}{6}}+\tau^{-\frac{3}{4}})\log\big(2+\Vert\zeta\Vert_{L_{\tau}^{\infty}(L^{2}\cap L^{\bar{p}}}\big)d\tau.
\end{equation}
Putting together \eqref{conss} and \eqref{logzeta}, we find that
$$\log\big(2+\Vert\zeta\Vert_{L_{\tau}^{\infty}(L^{2}\cap L^{\bar{p}})}\big)\le C_{0}(1+t^{\frac{19}{6}})+C_{0}\int_{0}^{t}(\tau^{\frac{13}{6}}+\tau^{-\frac{3}{4}})\log\big(2+\Vert\zeta\Vert_{L_{\tau}^{\infty}(L^{2}\cap L^{\bar{p}})}\big)d\tau.$$
Gronwall's inequality gives
$$\log\big(2+\Vert\zeta\Vert_{L_{\tau}^{\infty}(L^{2}\cap L^{\bar{p}})}\big)\le C_{0}(1+t^{\frac{19}{6}})e^{C_{0}(t^{\frac{19}{6}}+t^{\frac{1}{4}})}\le\Phi_{1}(t).$$
Therefore we get by using again \eqref{logzeta} that
$$\Vert x_{h}\rho\Vert_{L_{t}^{1}B^{0}_{\infty,1}}\le\Phi_{1}(t).$$
This yields in \eqref{conss} that 
$$\Vert\zeta(t)\Vert_{L^{2}\cap L^{\bar{p}}}\le\Phi_{2}(t).$$
Hence, it follows from \eqref{ah} that,
\begin{eqnarray*}
\Vert\zeta(t)\Vert_{L^{3,1}}\le\Phi_{2}(t).
\end{eqnarray*}
Then we have with $\omega=\omega_{\theta}e_\theta$ and by using \eqref{claspro},
$$\Vert\frac{\omega}{r}(t)\Vert_{L^{3,1}}\le\Vert\zeta(t)\Vert_{L^{3,1}}\le\Phi_{2}(t).$$
Thanks to \eqref{as}, we obtain 
$$\Vert\frac{v^r}{r}(t)\Vert_{L^\infty}\le\Phi_{2}(t).$$ 
This ends the proof of the proposition.
\end{proof}
Now, we will use the above estimates to obtain an bound for $\Vert\omega(t)\Vert_{L^\infty}.$ 
\begin{prop}\label{prop c} Under the same hypotheses of Proposition \ref{prop b} and if in addition $\omega^{0}\in L^\infty.$ Then we have for every $t \in \RR_{+},$
$$\Vert\omega(t)\Vert_{L^\infty}+\Vert\nabla\rho\Vert_{L_{t}^{1} L^\infty}\le\Phi_{4}(t).$$ We recall that $\Phi_{4}(t)$ does not depend on the viscosity.
\end{prop}
\begin{proof} Recall that the vorticity $\omega$ satisfies the equation
$$\partial_{t}\omega+v\cdot\nabla\omega-\nu\Delta\omega=\frac {v^r}{r}\omega+curl(\rho e_{z}).$$
 Applying the maximum principle and using Proposition \ref{prop b},
\begin{eqnarray*}
\Vert\omega(t)\Vert_{L^\infty}&\le& \Vert\omega^0\Vert_{L^\infty}+\int_0^t \Vert\frac{v^r}{r}(\tau)\Vert_{L^{\infty}}\Vert\omega(\tau)\Vert_{L^\infty}d\tau+ \int_0^t \Vert curl(\rho e_z)(\tau) \Vert_{L^\infty}d\tau\\
&\le& \Vert\omega^{0}\Vert_{L^\infty}+\int_{0}^{t}\Phi_{2}(\tau)\Vert\omega(\tau)\Vert_{L^\infty} d\tau+\int_0^t \Vert\nabla\rho(\tau)\Vert_{L^\infty} d\tau. 
\end{eqnarray*}
This implies by Gronwall inequality,
\begin{equation}\label{dc}
\Vert\omega(t)\Vert_{L^\infty}\le\bigg(\Vert\omega^{0}\Vert_{L^\infty}+\int_0^t \Vert\nabla\rho(\tau)\Vert_{L^\infty}d\tau\bigg)\Phi_{3}(t).
\end{equation}
It remains to estimate $\Vert\nabla\rho\Vert_{L_{t}^{1}L^{\infty}}.$ For this purpose we use Bernstein inequality for $\bar{p}>3,$ we obtain
\begin{eqnarray*}
\Vert\nabla\rho\Vert_{L_{t}^{1}L^{\infty}}&\le& \Vert\nabla\Delta_{-1}\rho\Vert_{L_t^1L^\infty}+\sum_{j\ge 0}\Vert\nabla\Delta_{j}\rho\Vert_{L_t^1L^\infty}\\
&\lesssim& \Vert\rho\Vert_{L_{t}^{1}L^{2}}+\sum_{j\ge0}2^{j(\frac{3}{\bar{p}}+1)}\Vert \Delta_{j}\rho\Vert_{L_{t}^{1}L^{\bar{p}}}\\
\end{eqnarray*}
Using now Proposition \ref{smooth}, Proposition \ref{prop a} and Bernstein inequality for $\bar{p}>3,$
\begin{eqnarray}\label{dd}
\nonumber \Vert\nabla\rho\Vert_{L_{t}^{1}L^{\infty}}&\lesssim&  \Vert\rho^{0}\Vert_{L^2}t+\sum_{j\ge0}2^{j(\frac{3}{\bar{p}}-1)}\Vert\rho^0 \Vert_{L^{\bar{p}}}\Big(1+(j+1)\int_0^t \Vert\omega(\tau)\Vert_{L^\infty} d\tau\Big)\\
\nonumber &+& \sum_{j\ge0}2^{j(\frac{3}{\bar{p}}-1)}\Vert\rho^0 \Vert_{L^{\bar{p}}} \int_0^t \Vert\nabla\Delta_{-1} v(\tau)\Vert_{L^\infty} d\tau\\
\nonumber &\lesssim& \Vert\rho^{0}\Vert_{L^2}t+\Vert\rho^0 \Vert_{L^{\bar{p}}}\Big(1+\int_0^t \Vert\omega(\tau)\Vert_{L^\infty} d\tau+\int_0^t \Vert v(\tau)\Vert_{L^2} d\tau\Big)\\
\nonumber &\lesssim& \Vert\rho^{0}\Vert_{L^2}t+\Vert\rho^0 \Vert_{L^{\bar{p}}}\Big(1+t\Vert v\Vert_{L_{t}^{\infty}L^2}+\int_0^t \Vert\omega(\tau)\Vert_{L^\infty} d\tau\Big)\\
\nonumber &\lesssim& \Vert\rho^{0}\Vert_{B_{2,1}^{0}\cap B_{\bar{p},1}^{0}}\Big(1+t+C_{0}t(1+t)+\int_0^t \Vert\omega(\tau)\Vert_{L^\infty} d\tau\Big)\\
&\le& C_{0}\Big(1+t^{2}+\int_0^t \Vert\omega(\tau)\Vert_{L^\infty} d\tau\Big).
\end{eqnarray}
Putting \eqref{dd} into \eqref{dc} and using Gronwall's inequality, we obtain
\begin{eqnarray*}
\Vert\omega\Vert_{L^\infty}&\lesssim& \bigg(\Vert\omega^0\Vert_{L^\infty}+C_0 \big(1+t^2+\int_0^t \Vert\omega(\tau)\Vert_{L^\infty} d\tau\big)\bigg)\Phi_{3}(t)\\
&\le& \Phi_{4}(t).
\end{eqnarray*}
This gives in \eqref{dd},
$$\Vert\nabla\rho\Vert_{L_{t}^{1} L^\infty}\le\Phi_{4}(t),$$
which is the desired result.
\end{proof}
Now, we will propagate globally in time the subcritical Sobolev regularities which is based on the estimate of $\Vert\nabla v(t)\Vert_{L^\infty}.$ 
 More precisely, we prove the following Proposition.
\begin{prop}\label{h1} Let $(v, \rho)$ be a smooth solution of the stratified system \eqref{a} with $\nu\ge 0$, and such that $(v^0, \rho^0)\in H^{s}\times H^{s-2}$ with $\frac{5}{2}<s.$ Then there exists $\Psi\in\mathcal{U}$ such that $(v^0, \rho^0)\in H^{s,\Psi}\times H^{s-2,\Psi}$ and for every $t \in \RR_+,$ 
$$\Vert v\Vert_{\widetilde{L}_{t}^{\infty}H^{s,\Psi}}+\Vert\rho\Vert_{\widetilde{L}_{t}^{\infty}H^{s-2,\Psi}}+\Vert\rho\Vert_{\widetilde{L}_{t}^{1}H^{s,\Psi}}\lesssim\Big(\Vert v^{0}\Vert_{H^{s,\Psi}}+\Vert\rho^{0}\Vert_{H^{s-2,\Psi}}(1+t)\Big)e^{C(\Vert\nabla v \Vert_{L_{t}^{1}L^{\infty}}+\Vert\nabla\rho\Vert_{L_{t}^{1} L^\infty})}.$$
If in addition, $\rho^{0}\in L^{m}$ with $m>6$ and $\vert x_{h}\vert^{2}\rho^{0}\in L^2,$ then we get for every $t\ge 0,$ 
\begin{equation*}
\Vert\nabla v(t)\Vert_{L^\infty}\le\Phi_{5}(t),\qquad\Vert v\Vert_{\widetilde{L}_{t}^{\infty}H^{s,\Psi}}+\Vert\rho\Vert_{\widetilde{L}_{t}^{\infty}H^{s-2,\Psi}} +\Vert\rho\Vert_{\widetilde{L}_{t}^{1}H^{s,\Psi}}\le\Phi_{6}(t).
\end{equation*}
The constants $C,$ $\Phi_{5}(t)$ and $\Phi_{6}(t)$ do not depend on the viscosity.
\end{prop}
\begin{Rema} 
From the Definition \ref{de1}, we observe that when the profile $\Psi$ is a nonnegative constant, then $H^{s,\Psi}=H^{s}$. In this case, we get the global persistence of the Sobolev regularities  
\begin{equation*}
\Vert\nabla v(t)\Vert_{L^\infty}\le\Phi_{5}(t),\qquad\Vert v\Vert_{\widetilde{L}_{t}^{\infty}H^{s}}+\Vert\rho\Vert_{\widetilde{L}_{t}^{\infty}H^{s-2}} +\Vert\rho\Vert_{\widetilde{L}_{t}^{1}H^{s}}\le\Phi_{6}(t).
\end{equation*}
Recall that $\Phi_{5}(t)$ and $\Phi_{6}(t)$ do not depend on the viscosity.
\end{Rema}
\begin{proof}
We localize in frequency the equation of the velocity, then we have for every $j \ge -1,$
$$\partial_{t}\Delta_{j} v+v\cdot\nabla\Delta_{j} v-\nu\Delta\Delta_{j} v+\nabla\Delta_{j}p=\Delta_{j}\rho e_{z}-[\Delta_{j},v\cdot\nabla]v.$$ 
Taking the $L^2$- scalar product of the above equation with $\Delta_{j}v$ and using H\"older inequality,
$$\frac{1}{2}\frac{d}{dt}\Vert\Delta_{j} v(t)\Vert^{2}_{L^2}+\nu\Vert\nabla\Delta_{j}v\Vert^{2}_{L^2}\le \Vert\Delta_{j} v(t)\Vert_{L^2}\Big(\Vert\Delta_{j} \rho(t)\Vert_{L^2}+\Vert[\Delta_{j},v\cdot\nabla]v(t)\Vert_{L^2}\Big).$$
Then,
$$\frac{d}{dt}\Vert\Delta_{j} v(t)\Vert_{L^2}\le \Vert\Delta_{j}\rho(t)\Vert_{L^2}+ \Vert[\Delta_{j},v\cdot\nabla]v(t)\Vert_{L^2}.$$
Integrating in time we obtain,
$$\Vert\Delta_{j} v(t)\Vert_{L^2}\le \Vert\Delta_{j} v^{0}\Vert_{L^2}+\Vert\Delta_{j}\rho\Vert_{L_{t}^{1} L^2}+ \Vert[\Delta_{j},v\cdot\nabla]v\Vert_{L_{t}^{1}L^2}.$$
Multiplying this inequality by $\Psi(j)2^{sj},$ taking the $\ell^{2}$-norm, we get
$$\Vert v\Vert_{\widetilde{L}_{t}^{\infty}H^{s,\Psi}}\le\Vert v^{0}\Vert_{H^{s,\Psi}}+\Vert\rho\Vert_{\widetilde{L}_{t}^{1}H^{s,\Psi}}+\Big(\Psi(j)2^{sj}\Vert[\Delta_{j},v\cdot\nabla]v\Vert_{L_{t}^{1}L^{2}}\Big)_{\ell^{2}}.$$
Combining Lemma \ref{le1} with Proposition \ref{pr4} to get,
\begin{eqnarray*} 
\Big(\Psi(j)2^{sj}\Vert[\Delta_{j},v\cdot\nabla]v\Vert_{L_{t}^{1}L^{2}}\Big)_{\ell^2} &\lesssim& \int_{0}^{t}\Big(\Psi(j)2^{sj}\Vert[\Delta_{j},v\cdot\nabla]v(\tau)\Vert_{L^{2}}\Big)_{\ell^2}d\tau\\
&\lesssim& \int_{0}^{t}\Vert\nabla v(\tau) \Vert_{L^\infty}\Vert v(\tau)\Vert_{H^{s,\Psi}}d\tau.
\end{eqnarray*}
Therefore we get,
\begin{equation}\label{q1}
\Vert v\Vert_{\widetilde{L}_{t}^{\infty}H^{s,\Psi}}\le\Vert v^{0}\Vert_{H^{s,\Psi}}+\Vert\rho\Vert_{\widetilde{L}_{t}^{1}H^{s,\Psi}}+C\int_{0}^{t}\Vert\nabla v(\tau) \Vert_{L^\infty}\Vert v(\tau)\Vert_{H^{s,\Psi}}d\tau.
\end{equation}
Now to estimate $\Vert\rho\Vert_{\widetilde{L}_{t}^{1}H^{s,\Psi}},$ we use Proposition \ref{pro xc} for $j\ge0$, 
\begin{equation*}
\Vert\Delta_{j}\rho\Vert_{L_{t}^{\infty}L^{2}}+2^{2j}\Vert\Delta_{j}\rho\Vert_{L_{t}^{1}L^{2}}\lesssim \Vert\Delta_{j}\rho^{0}\Vert_{L^2}+\Vert[\Delta_{j},v\cdot\nabla]\rho\Vert_{L_{t}^{1}L^{2}}.
\end{equation*}
Multiplying this last inequality by $\Psi(j)2^{j(s-2)},$ taking the $\ell^2$ norm, using H\"older inequality and Proposition \ref{prop a}-(b), we find
\begin{eqnarray*}
\Vert\rho\Vert_{\widetilde{L}_{t}^{\infty}H^{s-2,\Psi}}+\Vert\rho\Vert_{\widetilde{L}_{t}^{1}H^{s,\Psi}}&\le& \Vert\Delta_{-1}\rho\Vert_{L_{t}^{\infty}L^2}+\Vert\Delta_{-1}\rho\Vert_{L_{t}^{1}L^2} +\Vert\rho^{0}\Vert_{H^{s-2\Psi}}\\
&+& \Big\Vert\Psi(j)2^{j(s-2)}\Vert[\Delta_{j},v\cdot\nabla]\rho \Vert_{L_{t}^{1}L^{2}}\Big\Vert_{\ell^2}\\
&\le& C t\Vert\rho^{0}\Vert_{L^2}+C\Vert\rho^{0}\Vert_{H^{s-2,\Psi}}+ \Big\Vert\Psi(j)2^{j(s-2)}\Vert[\Delta_{j},v\cdot\nabla]\rho\Vert_{L_{t}^{1}L^{2}}\Big\Vert_{\ell^2}.
\end{eqnarray*}
Since $0<s-2,$ then using Proposition \ref{pr4}, we obtain that
$$\Big\Vert\Big(\Psi(j)2^{j(s-2)}\Vert[\Delta_{j},v\cdot\nabla] \rho\Vert_{L_{t}^{1}L^{2}}\Big)_{j}\Big\Vert_{\ell^2}\le C\int_{0}^{t}\big(\Vert\nabla v(\tau)\Vert_{L^\infty}\Vert\rho(\tau)\Vert_{H^{s-2,\Psi}}+\Vert\nabla\rho(\tau)\Vert_{L^\infty}\Vert v(\tau)\Vert_{H^{s-2,\Psi}}\big)d\tau.$$ 
Therefore by using the embeddings $H^{s,\Psi}\hookrightarrow H^{s-2,\Psi},$ we find 
\begin{eqnarray}\label{q2}
\nonumber \Vert\rho\Vert_{\widetilde{L}_{t}^{\infty}H^{s-2,\Psi}}+\Vert\rho\Vert_{\widetilde{L}_{t}^{1}H^{s,\Psi}}&\lesssim& \Vert\rho^{0}\Vert_{L^2}t+\Vert\rho^{0}\Vert_{H^{s-2,\Psi}}+\int_{0}^{t}\Vert\nabla\rho(\tau)\Vert_{L^\infty}\Vert v(\tau)\Vert_{\widetilde{L}_{\tau}^{\infty}H^{s,\Psi}}d\tau\\
&+& \int_{0}^{t}\Vert\nabla v(\tau)\Vert_{L^\infty}\Vert\rho(\tau)\Vert_{\widetilde{L}_{\tau}^{\infty}H^{s-2,\Psi}}d\tau.
\end{eqnarray}
Set $f(t):=\Vert v\Vert_{\widetilde{L}_{t}^{\infty}H^{s,\Psi}}+\Vert\rho\Vert_{\widetilde{L}_{t}^{\infty}H^{s-2,\Psi}}+\Vert\rho\Vert_{\widetilde{L}_{t}^{1}H^{s,\Psi}} $ and combining \eqref{q1} and \eqref{q2} with Gronwall's inequality, we obtain,
\begin{equation}\label{1q2}
f(t)\lesssim\Big(\Vert v^{0}\Vert_{H^{s,\Psi}}+\Vert\rho^{0}\Vert_{L^2}t+\Vert\rho^{0}\Vert_{H^{s-2,\Psi}}\Big)e^{C\big(\Vert\nabla v \Vert_{L_{t}^{1}L^{\infty}}+\Vert\nabla\rho \Vert_{L_{t}^{1}L^{\infty}}\big)}.
\end{equation}
To estimate the term $\Vert\nabla\rho\Vert_{L_{t}^{1}L^\infty},$ we use Proposition \ref{prop c} and to estimate the Lipschitz norm of the velocity, 
we use the classical logarithmic estimate: for $s>\frac{5}{2},$
\begin{eqnarray*}
\Vert\nabla v\Vert_{L^\infty}&\lesssim& \Vert v\Vert_{L^2}+\Vert\omega\Vert_{L^\infty}\log(e+\Vert v \Vert_{H^s})\\
&\lesssim& \Vert v\Vert_{L^2}+\Vert\omega\Vert_{L^\infty}\log(e+\Vert v\Vert_{\widetilde{L}_{t}^{\infty}H^{s,\Psi}}),
\end{eqnarray*}
where we have used in the last line the embedding $H^{s,\Psi}\hookrightarrow H^{s}.$ 
Combining this estimate with \eqref{1q2}, Proposition \ref{prop a}-(a) and Proposition \ref{prop c}, we get
$$\Vert\nabla v\Vert_{L^\infty}\le\Phi_{4}(t)\bigg(1+t+\int_{0}^{t}\Vert\nabla v(\tau)\Vert_{L^\infty}d\tau\bigg).$$
This gives by Gronwall's inequality
$$\Vert\nabla v(t) \Vert_{L^\infty}\le\Phi_{5}(t).$$
Plugging this estimate into \eqref{1q2}, we obtain finally
\begin{eqnarray*}
f(t):=\Vert v\Vert_{\widetilde{L}_{t}^{\infty}H^{s,\Psi}}+\Vert\rho\Vert_{\widetilde{L}_{t}^{\infty}H^{s-2,\Psi}}+\Vert\rho\Vert_{\widetilde{L}_{t}^{1}H^{s,\Psi}}\le\Phi_{6}(t).
\end{eqnarray*}
This end the proof of the proposition.
\end{proof}
\subsection{Inviscid limit}\label{A5}
we will prove that the family $(v_{\nu},\rho_\nu)_{\nu>0}$ is converges strongly in $L_{T}^{\infty}H^{s}\times L_{T}^{\infty}H^{s-2}$ to the solution $(v,\rho)$ of the Euler-stratified system \eqref{fd1} as $\nu\to 0.$ More precisely, we prove the following proposition.
\begin{prop}\label{pro 98}
Let $s>\frac{5}{2},$ $v^{0}$ be an axisymmetric divergence-free vector field such that $v^{0}\in H^{s}$ and $\rho^{0}\in \chi_{m}^{s}$ with $6<m.$ Then the solution $(v_\nu,\rho_\nu)$ to the system \eqref{a} converges strongly as $\nu\to 0$ to the unique solution $(v,\rho)$ of the system \eqref{fd1} in $L_{loc}^{\infty}(\RR_+;H^{s})\times L_{loc}^{\infty}(\RR_+;H^{s-2}).$\\ 
More precisely, there exists  $\Psi\in\mathcal{U}_{\infty}$ depending on the profile of the initial data and such that for every $T>0$
$$\Vert v_{\nu}-v\Vert_{L_{T}^{\infty}H^{s}}+\Vert\rho_{\nu}-\rho\Vert_{L_{T}^{\infty}H^{s-2}}\le\Big(\sqrt{\nu}+\frac{1}{\Psi(\log(\frac{1}{\nu}))}\Big)\Phi_{7}(T).$$
\end{prop}
\begin{proof}
We will proceed in two steps. In the first one, we prove that for any fixed $T>0,$ the family $(v_\nu,\rho_\nu)_{\nu}$ converges strongly in $L_{T}^{\infty}L^{2}$ when $\nu\to 0,$ to the solution $(v,\rho)$ of the system \eqref{fd1} with initial data $(v^0,\rho^0).$ In the second step, we will show how to get the strong convergence in the Sobolev spaces $L_{T}^{\infty}H^{s}\times L_{T}^{\infty}H^{s-2},$ with $s>\frac52.$\\
We set $$W_{\nu}:=v_{\nu}-v,\;\;\Pi_{\nu}=p_{\nu}-p\;\;\textnormal{and}\;\,\eta_{\nu}=\rho_{\nu}-\rho.$$
Then we obtain the equations:
\begin{equation}\label{fp1}
\left\{ \begin{array}{ll} 
\partial_{t}W_{\nu}+v_{\nu}\cdot\nabla W_{\nu}+W_{\nu}\cdot\nabla v-\nu\Delta W_{\nu}+\nabla\Pi_{\nu}=\nu\Delta v+\eta_{\nu}\,e_{z}\\ 
\partial_{t}\eta_{\nu}+v_{\nu}\cdot\nabla\eta_{\nu}-\Delta\eta_{\nu}=-W_{\nu}\cdot\nabla\rho\\
\textnormal{div}\,W_{\nu}=0\\
(W_{\nu},\eta_{\nu})_{| t=0}=0.  
\end{array} \right.
\end{equation} 
First, we take the $L^{2}$ inner product of the first equation of \eqref{fp1} with $W_{\nu},$ integrating by parts and using H\"older inequality, we get
$$\frac{1}{2}\frac{d}{dt}\Vert W_{\nu}(t)\Vert^{2}_{L^2}+\nu\Vert\nabla W_{\nu}\Vert^{2}_{L^2}\le\nu\Vert\Delta v\Vert_{L^2}\Vert W_{\nu}\Vert_{L^2}+\Vert\nabla v\Vert_{L^\infty}\Vert W_{\nu}\Vert^{2}_{L^2}+\Vert\eta_{\nu}\Vert_{L^2}\Vert W_{\nu}\Vert_{L^2}.$$
 This gives,
 $$\frac{d}{dt}\Vert W_{\nu}(t)\Vert_{L^2}\le\nu\Vert\Delta v\Vert_{L^2}+\Vert\nabla v\Vert_{L^\infty}\Vert W_{\nu}\Vert_{L^2}+\Vert\eta_{\nu}\Vert_{L^2}.$$
Integrating in time this last inequality , we obtain
\begin{equation*}
\Vert W_{\nu}(t)\Vert_{L^2}\le\nu\Vert\Delta v\Vert_{L_{t}^{1}L^2}+\Vert\eta_{\nu}\Vert_{L_{t}^{1}L^2}+\int_{0}^{t}\Vert\nabla v(\tau)\Vert_{L^\infty}\Vert W_{\nu}(\tau)\Vert_{L^2}d\tau.
\end{equation*}
From the inequality
$$\Vert\Delta v\Vert_{L^2}\le C\Vert v\Vert_{H^s},\;\;s\ge 2$$
and by using Gronwall inequality, Proposition \ref{h1}, we find
\begin{eqnarray}\label{fp6}
\nonumber \Vert W_{\nu}\Vert_{L_{t}^{\infty}L^2}&\lesssim& \Big(\nu\Vert\Delta v\Vert_{L_{t}^{1}L^2}+\Vert\eta_{\nu}\Vert_{L_{t}^{1}L^2}\Big)e^{\Vert\nabla v\Vert_{L_{t}^{1}L^{\infty}}}\\
\nonumber &\le& \Big(\nu t\Vert v\Vert_{\widetilde{L}_{t}^{\infty}H^s}+\Vert\eta_{\nu}\Vert_{L_{t}^{1}L^{2}}\Big)\Phi_{6}(t)\\
&\le& \Big(\nu+\Vert\eta_{\nu}\Vert_{L_{t}^{1}L^{2}}\Big)\Phi_{6}(t).
\end{eqnarray}
It remains to estimate $\Vert\eta_{\nu}\Vert_{L_{t}^{1}L^{2}}.$ For this purpose, 
we apply the maximum principle to the second equation of \eqref{fp1}, we get
\begin{eqnarray}\label{Maxi}
\nonumber \Vert\eta_{\nu}(t)\Vert_{L^2}&\le& \int_{0}^{t}\Vert W_{\nu}\cdot\nabla\rho(\tau)\Vert_{L^2}d\tau\\
\nonumber &\lesssim& \Vert W_{\nu}\Vert_{L_{t}^{\infty}L^2}\Vert\nabla\rho\Vert_{L_{t}^{1}L^\infty}\\
&\le& \Phi_{4}(t)\Vert W_{\nu}\Vert_{L_{t}^{\infty}L^2},
\end{eqnarray} 
where we have used Proposition \ref{prop c}. Putting now \eqref{Maxi} into \eqref{fp6} and using Gronwall inequality, we get for all $t\in[0,T]$
\begin{eqnarray}\label{fp2}
\nonumber \Vert W_{\nu}\Vert_{L_{t}^{\infty}L^2}&\le& \Big(\nu+\int_{0}^{t}\Phi_{4}(\tau)\Vert W_{\nu}\Vert_{L_{\tau}^{\infty}L^2}d\tau\Big)\Phi_{6}(t)\\
&\le& \nu\Phi_{7}(t).
\end{eqnarray}
This gives in view of \eqref{Maxi} that,
$$\Vert\eta_{\nu}\Vert_{L_{t}^{\infty}L^2}\le\nu\Phi_{7}(t)\;\;\forall t\in[0,T]..$$
Therefore 
$$\Vert W_{\nu}\Vert_{L_{t}^{\infty}L^2}+\Vert\eta_{\nu}\Vert_{L_{t}^{\infty}L^2}\le\nu\Phi_{7}(t)\;\;\forall t\in[0,T].$$
This achieves the proof of the strong convergence in $L_{loc}^{\infty}(\RR_+;L^2).$\\
Let us now turn to the proof of the strong convergence in the Sobolev spaces. Let $M\in\NN,$ that will be chosen later, then by definition of Sobolev space we have $\forall t\in\RR_+,$
\begin{eqnarray*}
\Vert(v_{\nu}-v)(t)\Vert^{2}_{H^s}&=&\sum_{q\le M}2^{2qs}\Vert\Delta_{q}(v_{\nu}-v)(t)\Vert^{2}_{L^2}+\sum_{q>M}2^{2qs}\Vert\Delta_{q}(v_{\nu}-v)(t)\Vert^{2}_{L^2}\\
&\lesssim& 2^{2Ms}\Vert(v_{\nu}-v)(t)\Vert^{2}_{L^2}+\frac{1}{\Psi^{2}(M)}\sum_{q>M}\Psi^{2}(q)2^{2qs}\Big(\Vert\Delta_{q}v_{\nu}(t)\Vert_{L^2}+\Vert\Delta_{q}v(t)\Vert_{L^2}\Big)^{2}\\
&\lesssim& 2^{2Ms}\Vert W_{\nu}(t)\Vert^{2}_{L^2}+\frac{1}{\Psi^{2}(M)}\Big(\Vert v_{\nu}(t)\Vert^{2}_{H^{s,\Psi}}+\Vert v(t)\Vert^{2}_{H^{s,\Psi}}\Big).
\end{eqnarray*}
We have used the fact that the profile $\Psi$ is nondecreasing. Now, we use \eqref{fp2} and Proposition \ref{h1}, to get
$$\Vert(v_{\nu}-v)(t)\Vert^{2}_{H^s}\le\Big(2^{2Ms}\nu^{2}+\frac{1}{\Psi^{2}(M)}\Big)\Phi_{7}(t).$$
It is enough to choose $M$ such that
$$e^{2Ms}\approx\frac{1}{\nu}.$$
Therefore we obtain that
$$\Vert(v_{\nu}-v)(t)\Vert^{2}_{H^s}\le\Big(\nu+\frac{1}{\Psi^{2}(\frac{1}{2s}\log(\frac{1}{\nu}))}\Big)\Phi_{7}(t).$$
Similarly for $\Vert(\rho_{\nu}-\rho)(t)\Vert_{H^{s-2}},$ we obtain finally
\begin{eqnarray*}
\Vert(\rho_{\nu}-\rho)(t)\Vert^{2}_{H^{s-2}}&\le&\Big(\nu+\frac{1}{\Psi^{2}(\frac{1}{2(s-2)}\log(\frac{1}{\nu}))}\Big)\Phi_{7}(t)\\
&\le&\Big(\nu+\frac{1}{\Psi^{2}(\frac{1}{2s}\log(\frac{1}{\nu}))}\Big)\Phi_{7}(t).
\end{eqnarray*}
In the last line, we have used the fact that the profile $\Psi$ is nondecreasing. Now for any $\lambda>0,$ the function defined by $\Psi_{\lambda}(x):=\Psi(\lambda x)$ belongs to the same class $\mathcal{U}_\infty.$ Therefore, we get by modify $\Psi:$
$$\Vert(v_{\nu}-v)(t)\Vert_{H^{s}}+\Vert(\rho_{\nu}-\rho)(t)\Vert_{H^{s-2}}\le\Big(\sqrt{\nu}+\frac{1}{\Psi(\log(\frac{1}{\nu}))}\Big)\Phi_{7}(t).$$ 
It follows that 
$$\Vert v_{\nu}-v\Vert_{L_{T}^{\infty}H^s}+\Vert\rho_{\nu}-\rho\Vert_{L_{T}^{\infty}H^{s-2}}\to 0\quad\textnormal{as}\;\;\nu\to 0.$$  
This achieved the desired result of the proposition.
\end{proof}

\end{document}